\documentclass[10pt]{article}
\usepackage{lmodern}
\usepackage{amsmath}
\usepackage[T1]{fontenc}
\usepackage[utf8]{inputenc}
\usepackage{authblk}
\usepackage{amsfonts}
\usepackage{graphicx}
\usepackage{rotating}
\usepackage{amssymb}
\usepackage[english]{babel}
\usepackage{color}
\usepackage{amsthm}
\usepackage{graphicx}
\usepackage{tkz-euclide}
\tikzset{elegant/.style={smooth,thick,samples=50,cyan}}
\usepackage{color,tikz}
\usetikzlibrary{patterns}
\usetikzlibrary{decorations.pathreplacing,calc}
\usepackage{rotating}
\usepackage{mathrsfs}
\usepackage{makecell}
\usepackage{microtype}
\usepackage{enumerate}
\usepackage{mathscinet}
\usepackage{array}
\usepackage{multirow}
\usepackage{booktabs}
\usepackage{enumerate}
\usepackage[cal=boondoxo,bb=ams]{mathalfa}
\usepackage{hyperref}
\hypersetup{hidelinks}
\usepackage{titlesec}
\newtheorem{theorem}{Theorem}[section]
\newtheorem{prop}{Proposition}[section]

\newtheorem{coro}{Corollary}[section]
\newtheorem{remark}{Remark}[section]

\newcommand{\ml}{\mathcal}
\newcommand{\mb}{\mathbb}

\DeclareMathOperator{\intt}{int}
\DeclareMathOperator{\extt}{ext}
\DeclareMathOperator{\bdd}{bdd}
\DeclareMathOperator{\divv}{div}
\DeclareMathOperator{\rott}{rot}

\title{Large-time asymptotic behaviors for the classical thermoelastic system}
\author[1]{Wenhui Chen\thanks{Wenhui Chen (wenhui.chen.math@gmail.com)}}
\affil[1]{School of Mathematics and Information Science, Guangzhou University, 510006 Guangzhou, China}
\author[2]{Hiroshi Takeda\thanks{Hiroshi Takeda (h-takeda@fit.ac.jp)}}
\affil[2]{Department of Intelligent Mechanical Engineering, Faculty of Engineering, Fukuoka Institute of Technology,  811-0295 Fukuoka, Japan}

\date{}
\begin{document}

\maketitle
\begin{abstract}
	\medskip
In this paper, we study the classical thermoelastic system with Fourier's law of heat conduction in the whole space $\mathbb{R}^n$ when $n=1,2,3$, particularly, asymptotic profiles for its elastic displacement as large-time. We discover optimal growth estimates of the elastic displacement when $n=1,2$, whose growth rates coincide with those for the free wave model, whereas when $n=3$ the optimal decay rate is related to the Gaussian kernel. Furthermore, under a new condition for weighted datum, the large-time optimal leading term is firstly introduced by the combination of diffusion-waves, the heat kernel and singular components. We also illustrate a second-order profile of solution by diffusion-waves with weighted $L^1$ datum  as a by-product. These results imply that wave-structure large-time behaviors hold only for the one- and two-dimensional thermoelastic systems.\\
	
	
	\noindent\textbf{AMS Classification (2020)} 35G40, 35B40, 35Q79
	
\end{abstract}
\fontsize{12}{15}
\selectfont
\section{Introduction}
It is well-known that the reciprocal actions between elastic stresses and thermal behaviors including
temperature difference of an elastic heat conductive media are effectively described by the thermoelastic systems mathematically (see \cite{Chandrasekharaiah=1998,Green-Naghdi=1991,Jiang-Racke=2000} and references therein). The classical model of thermoelasticity (see an extensive review \cite{Chandrasekharaiah=1998}) is constructed by an elasticity, e.g. the isotropic elastic waves, coupled with Fourier's law of heat conduction
\begin{align*}
{q}=-\kappa\nabla\theta 
\end{align*}
with the heat flux ${q}={q}(t,x)\in\mb{R}^n$ and the temperature (relative to some reference temperature) $\theta=\theta(t,x)\in\mb{R}$, where the positive constant $\kappa$ denotes the thermal conductivity. In other words, the classical thermoelastic model is a hyperbolic-parabolic coupled system. Let us consider the Cauchy problem for the classical thermoelastic system in the physical dimensions $n=1,2,3$, namely,
\begin{align}\label{Linear-Thermoelastic-Type}
	\begin{cases}
		u_{tt}-a^2\Delta u-(b^2-a^2)\nabla\divv u+\gamma_1\nabla\theta=0,\\
		\theta_{t}-\kappa\Delta\theta+\gamma_2\divv u_{t}=0, \\
	u(0,x)=u_0(x),\ u_t(0,x)=u_1(x),\ \theta(0,x)=\theta_0(x),
\end{cases}
\end{align}
with $(t,x)\in\mb{R}_+\times\mb{R}^n$, where two unknowns $u=u(t,x)\in\mb{R}^n$ and $\theta=\theta(t,x)\in\mb{R}$ stand for the elastic displacement and the temperature difference to the equilibrium state, respectively. The speeds for propagation of the longitudinal P-wave and of the transverse S-wave are denoted by $b$ and $a$, individually, fulfilling $b>a>0$, whose combinations $a^2$ and $b^2-2a^2$ are the well-known Lam\'e constants. Some physical properties of the underlying isotropic medium for \eqref{Linear-Thermoelastic-Type} are described by the thermal conductivity $\kappa>0$ and the thermoelastic coupling coefficients $\gamma_1,\gamma_2$ such that $\gamma_1\gamma_2>0$.

 We recall now some historical background to the classical thermoelasticity 
 \begin{align}\label{General}
 	\begin{cases}
 		u_{tt}-a^2\Delta u-(b^2-a^2)\nabla\divv u+\gamma_1\nabla\theta=0,\\
 		\theta_{t}-\kappa\Delta\theta+\gamma_2\divv u_{t}=0.
 	\end{cases}
 \end{align}
  Actually, the theory of thermoelasticity is very classical and was first founded by J.M.C. Duhamel, K.E. Neumann and W. Thomson, etc. in  18th century. Therewith, the pioneering paper \cite{LL=1953} applied the classical thermodynamics methods to deduce the coupled system of thermoelasticity. In the homogeneous and isotropic medium with vanishing external body forces as well as heats, the classical thermoelastic system \eqref{General} can be derived in general. In recent thirty years, the classical thermoelastic systems  \eqref{General} in bounded or unbounded domains have caught a lot of attentions (see \cite{Dafermos=1968,Racke-1987,Racke-1990,Racke-1990-02,Racke-Shibata=1991,Munoz=1992,Munoz-Racke=1996,Racke-Wang=1998,Reissig-Wang=1999,Jiang-Racke=2000,Wang=2003,Jachmann-Reissig=2009,Jachmann-Wirth=2010} and references therein).  According to the theme of this work, we just briefly introduce the
progressive progress in the corresponding Cauchy problem \eqref{Linear-Thermoelastic-Type}. By employing the Helmholtz decomposition, the solution may be split into $$u=u^{p_0}+ u^{s_0},$$ where the solenoidal part $u^{s_0}$ solves the well-studied wave equation (see \eqref{us-waves} soon afterwards) and the potential part $u^{p_0}$  fulfills
\begin{align}\label{up-coupled}
	\begin{cases}
		u^{p_0}_{tt}-b^2\Delta u^{p_0}+\gamma_1\nabla\theta=0,\\
		\theta_{t}-\kappa\Delta\theta+\gamma_2\divv u^{p_0}_{t}=0,\\
		u^{p_0}(0,x)=u_0^{p_0}(x),\ u^{p_0}_t(0,x)=u_1^{p_0}(x),\ \theta(0,x)=\theta_0(x),
	\end{cases}
\end{align}
carrying  $\rott u^{p_0}=0$ with $(t,x)\in\mb{R}_+\times\mb{R}^n$ for $n=1,2,3$. The authors of \cite{Jiang-Racke=2000} constructed the standard energy $(b\divv u^{p_0},u^{p_0}_t,\theta)\in\mb{R}^{n+2}$ and employed energy methods for the first-order coupled system associated with anti-symmetric properties of coefficient matrix to derive 
\begin{align}\label{2.1}
	\|(b\divv u^{p_0},u^{p_0}_t,\theta)(t,\cdot)\|_{(L^2)^{n+2}}\lesssim (1+t)^{-\frac{n}{4}}\|(b\divv u^{p_0}_0,u^{p_0}_1,\theta_0)\|_{(L^2\cap L^1)^{n+2}}.
\end{align}
Later, the paper \cite{Jachmann-Reissig=2009} applied the so-called \emph{diagonalization procedure} (it may decouple the system and search for the dominant parts of vector unknown in local zones) for a suitable micro-energy $(u^{p_0}_t\pm ib|D|u^{p_0},\theta)$ in which the operator $|D|$ owns the symbol $|\xi|$. Applying WKB analysis and multi-steps diagonalization method, the authors derived $L^2$ well-posedness, propagation of singularities, $L^p-L^q$ decay estimates and parabolic-type diffusion phenomenon. Particularly, the $(L^2\cap L^1)-L^2$ estimate conforms to the same decay rate as the one in \eqref{2.1}. For these reasons, in the framework of energy unknowns, the thermoelastic system \eqref{up-coupled} has parabolic-type decay properties whose decay rate comes from the Gaussian kernel $\ml{F}^{-1}_{\xi\to x}(\mathrm{e}^{-c|\xi|^2t})$.

 As we mentioned in the above, there are a lot of related works begun from the past century in terms of the classical thermoelasticity \eqref{Linear-Thermoelastic-Type}. Some suitable energy terms, e.g. the standard energy $(b\divv u^{p_0}, u_t^{p_0},\theta)$ in \cite[Equation (4.65)]{Jiang-Racke=2000}, decay polynomially with parabolic-type in the $L^2$ framework, and some asymptotic profiles of these energy terms fulfill parabolic-structure reference systems, e.g. diffusion phenomena stated in \cite{Jachmann-Reissig=2009}. Consequently, some natural questions are:
 \begin{itemize}
 	\item Whether or not one can describe more detailed information of the elastic displacement $u$ or $u^{p_0}$ for large-time?
 	\item How does the elastic waves part influence on the large-time behaviors?
 \end{itemize}
 It will play a crucial role when a model equips some nonlinear terms containing $u$ itself, for instance, \cite{Racke-Wang=1998,Reissig-Wang=1999,Kirane-Tatar=2001,Qin-Rivera=2004} considered the nonlinear source term $\ml{N}(u)$ appearing on the first equation of \eqref{Linear-Thermoelastic-Type}. However, to the best of knowledge of authors, the optimal estimates and optimal leading terms of the elastic displacement in thermoelasticity are still open even for the irrotational situation. Here, the optimality is guaranteed by the same behaviors of upper bound and lower bound estimates for large-time (see, for example, \cite{Ikehata=2014,Ikehata-Onodera=2017,Michihisa=2021} for  the viscoelastic damped waves). We will partially give answers to the above questions by deep understanding of the effect of thermal damping generated by Fourier's law.

The main purpose of this work is to study asymptotic behaviors of the potential part $u^{p_0}$ to the Cauchy problem \eqref{up-coupled} because of the well-established properties for the solenoidal part $u^{s_0}$.  Note that $u^{p_0}:=(u^{1,p_0},\cdots,u^{n,p_0})\in\mb{R}^n$ with $x\in\mb{R}^n$, and $u^{k,p_0}$ is a  element among them. By reducing the hyperbolic-parabolic coupled system to the third-order (in time) evolution equation with respect to $u^{k,p_0}$, we employ WKB analysis and Fourier analysis to characterize the optimal large-time estimates
\begin{align*}
	\|u^{k,p_0}(t,\cdot)\|_{L^2}^2\simeq\begin{cases}
		t&\mbox{if}\ \ n=1,\\
		\ln t&\mbox{if}\ \ n=2,\\
		t^{-\frac{1}{2}}&\mbox{if}\ \ n=3,
	\end{cases}
\end{align*}
with $L^1 \cap L^{2}$ datum if
\begin{align*}
|P_{u_1^{k,p_0}}|+|P_{\theta_0}|\neq0	
\end{align*}
carrying the mean $P_f=\int_{\mb{R}^n}f(x)\mathrm{d}x$, where the solution grows polynomially when $n=1$ and  logarithmically when $n=2$, but decay polynomially in higher-dimensions. Particularly, due to the identical growth rates for the thermoelastic system and the wave equation when $n=1,2$, we claim the decisive part of the one- and two-dimensional thermoelastic system \eqref{Linear-Thermoelastic-Type} is the waves part as $t\gg1$, whereas Fourier's law of heat conduction exerts crucial influence when $n=3$ (see the detail explanation in Remark \ref{Rem-01}). This wave-type large-time behavior is the new discovery, and different from the parabolic-type decay properties in the previous literature of thermoelasticity. One of the main difficulties is to understand interplay among dissipative part, oscillating part and the singularity (as $|\xi|\to0$) for the next multiplier: 
\begin{align*}
\ml{M}_n(t,|\xi|):=\frac{1}{|\xi|}\left(\mathrm{e}^{-\beta_0|\xi|^2t}-\cos(\beta_1|\xi|t)\mathrm{e}^{-\beta_2|\xi|^2t}\right)
\end{align*}
in the $L^2$ norm, where the positive constants are
\begin{align*}
	 \beta_0=\frac{\kappa b^2}{b^2+\gamma_1\gamma_2},\ \ \beta_1=\sqrt{b^2+\gamma_1\gamma_2} \ \  \mbox{and} \ \ \beta_2=\frac{\kappa\gamma_1\gamma_2}{2(b^2+\gamma_1\gamma_2)}.
\end{align*}

Furthermore, introducing the leading term $\varphi^k=\varphi^k(t,x)$ such that
\begin{align*}
\varphi^k:=\ml{F}_{\xi\to x}^{-1}\left(\frac{\sin(\beta_1|\xi|t)}{\beta_1|\xi|}\mathrm{e}^{-\beta_2|\xi|^2t}\right)P_{u_1^{k,p_0}} -\frac{i\gamma_1}{\beta_1^2}\ml{R}_k\ml{F}^{-1}_{\xi\to x}\left[\frac{1}{|\xi|}\left(\cos(\beta_1|\xi|t)\mathrm{e}^{-\beta_2|\xi|^2t}-\mathrm{e}^{-\beta_0|\xi|^2t}\right)\right]P_{\theta_0} 
\end{align*}
with  the Riesz transform $\ml{R}_k$, we derive the optimal decay estimates of the error term
\begin{align*}
\|u^{k,p_0}(t,\cdot)-\varphi^k(t,\cdot)\|_{L^2}^2\simeq t^{-\frac{n}{2}}
\end{align*}
for $t\gg1 $ and $n=1,2,3$ if 
\begin{align*}
|P_{u_0^{k,p_0}}|+  |P_{u_1^{k,p_0}}| +|P_{\theta_0}|  + |M_{u_1^{k,p_0}}|+ |M_{\theta_{0}}| \neq0
\end{align*}
carrying the weighted mean $M_f=\int_{\mb{R}^n}xf(x)\mathrm{d}x$. That is to say that the function $\varphi^k(t,x)$, which is the combination of diffusion-waves and heat kernel with the singular component (for small frequencies), is the optimal leading term for the thermoelastic system. As a by-product, we also investigate a second-order asymptotic profile of solution via higher-order diffusive-waves and heat model in the $L^2$ framework. To sum up, novel first- and second-order asymptotic profiles for large-time characterized by diffusion-waves (instead of parabolic structure reference systems in the previous researches) for the thermoelastic system are discovered.

To end this paper, we state some concluding remarks related to other models of thermoelasticity in Section \ref{Sect-Final}, and a new proof of optimal growth estimates for the free wave equation in Appendix \ref{Appendix-A}, which improves the assumption of initial datum in \cite[Theorems 1.1, and 1.2]{Ikehata=2022-wave}.

\medskip
\noindent\textbf{Notations:} Let us introduce some notations that will be used in this paper. We take the following zones localizing in the Fourier space:
\begin{align*}
	\ml{Z}_{\intt}(\varepsilon_0)&:=\{\xi\in\mb{R}^n:\ |\xi|\leqslant\varepsilon_0\ll1\},\\
	\ml{Z}_{\bdd}(\varepsilon_0,N_0)&:=\{\xi\in\mb{R}^n:\ \varepsilon_0\leqslant |\xi|\leqslant N_0\},\\  
	\ml{Z}_{\extt}(N_0)&:=\{\xi\in\mb{R}^n:\ |\xi|\geqslant N_0\gg1\}.
\end{align*}
Moreover, the cut-off functions $\chi_{\intt}(\xi),\chi_{\bdd}(\xi),\chi_{\extt}(\xi)\in \mathcal{C}^{\infty}$ having their supports in their corresponding zones $\ml{Z}_{\intt}(\varepsilon_0)$, $\ml{Z}_{\bdd}(\varepsilon_0/2,2N_0)$ and $\ml{Z}_{\extt}(N_0)$, respectively, such that
\begin{align*}
	\chi_{\bdd}(\xi)=1-\chi_{\intt}(\xi)-\chi_{\extt}(\xi)\ \ \mbox{for all}\ \ \xi \in \mb{R}^n.
\end{align*}
The symbol of pseudo-differential operator $|D|$ is denoted by $|\xi|$.

The notation $f\lesssim g$ means that there exists a positive constant $C$ fulfilling $f\leqslant Cg$, which may be changed in different places, analogously, for $f\gtrsim g$. Furthermore, the asymptotic relation  $f\simeq  g$ holds if and only if $f\lesssim g$ and $f\gtrsim g$ concurrently. We take the notation $\circ$ as the inner product in Euclidean space, i.e. $f\circ g:=\langle f,g\rangle $ for $f,g\in\mb{R}^n$.

Let us recall the weighted $L^1$ space as follows:
\begin{align*}
	L^{1,1}:=\left\{f\in L^1 \ \big|\ \|f\|_{L^{1,1}}:=\int_{\mb{R}^n}(1+|x|)|f(x)|\mathrm{d}x<\infty \right\}.
\end{align*}
The (weighted) means of a summable function $f$ are denoted by
\begin{align*}
	\mb{R}\ni P_f:=\int_{\mb{R}^n}f(x)\mathrm{d}x	\ \ \mbox{and}\ \ \mb{R}^n\ni M_f:=\int_{\mb{R}^n}xf(x)\mathrm{d}x.
\end{align*}
To complete the introduction, we take the following time-dependent function:
\begin{align}\label{Decay-fun}
	\ml{A}_n(t):=\begin{cases}
		\sqrt{t}&\mbox{if}\ \ n=1,\\
		\sqrt{\ln t}&\mbox{if}\ \ n=2,\\
		t^{-\frac{1}{4}}&\mbox{if}\ \ n= 3,
	\end{cases}
\end{align}
to be the growth ($n=1,2$) or decay ($n=3$) rates later.

\section{Main results}
\subsection{Pretreatments by the Helmholtz decomposition}\label{Subsec-Helmholtz}
Before stating the main results of this work, let us simplify the model \eqref{Linear-Thermoelastic-Type} and turn to our essential target hyperbolic-parabolic coupled system \eqref{up-coupled}, in which the one-dimensional case can be trivially got without using this approach. According to the Helmholtz decomposition
\begin{align*}
L^2=\overline{\nabla H^1}\oplus\ml{D}_0 \ \ \mbox{for}\ \ n=2,3,
\end{align*}
with the function spaces
\begin{align*}
\nabla H^1&:=\left\{\nabla\psi\ |\  \psi\in H^1\right\},\\
\ml{D}_0&:=\left\{u\in L^2\ |\ (\nabla\phi,u)_{L^2}=0\ \ \mbox{with}\ \ \forall \phi\in \ml{C}_0^{\infty}\right\},
\end{align*}
and thus by such technique the solution $u=u(t,x)$ to \eqref{Linear-Thermoelastic-Type} can be decomposed into a potential part and a solenoidal part such that
\begin{align}\label{Helm-solution}
	u=u^{p_0}+ u^{s_0}.
\end{align}
Here, the vector $u^{p_0}=u^{p_0}(t,x)$ is rotation-free and $u^{s_0}=u^{s_0}(t,x)$ is divergence-free in a weak sense. We clarify the nomenclature: \emph{the potential solution} by the unknown $u^{p_0}$ for the sake of briefness. Applying the identity
\begin{align*}
\nabla\divv u=\nabla\times(\nabla\times u)+\Delta u	
\end{align*}
in two- and three-dimensions, we are able to decompose \eqref{Linear-Thermoelastic-Type} into the wave model
\begin{align}\label{us-waves}
	\begin{cases}
	u^{s_0}_{tt}-a^2\Delta u^{s_0}=0,\\
	u^{s_0}(0,x)=u^{s_0}_0(x),\ u^{s_0}_t(0,x)=u^{s_0}_1(x),
	\end{cases}
\end{align}
and the hyperbolic-parabolic coupled system \eqref{up-coupled} with $(t,x)\in\mb{R}_+\times\mb{R}^n$, under the conditions $\divv u^{s_0}=0$ as well as $\rott u^{p_0}=0$. Although the last treatment were considered when $n=2,3$, one may notice that the one-dimensional coupled system \eqref{Linear-Thermoelastic-Type} is  exactly the same as the coupled system \eqref{up-coupled} when $n=1$. Therefore, in the one-dimensional case, we may understand $u=u^{p_0}$.

 The studies for linear wave equation \eqref{us-waves} are well-known, for example, the author of \cite[Theorems 1.1 and 1.2]{Ikehata=2022-wave} stated the optimal estimates for the weighted $L^1$ data, or the improved result in Corollary \ref{Coro-Appendix} requiring $L^1$ regularity only for the second data.  Again, our main task in this paper, consequently, will be immediately turned into:
\begin{center}
\it  Asymptotic behaviors for the hyperbolic-parabolic coupled system \eqref{up-coupled} when $n=1,2,3$.
\end{center}
Particularly, large-time asymptotic profiles including some optimal estimates, optimal leading terms as well as second-order approximations for the potential solution $u^{p_0}$ with some weighted $L^1$ datum are of interest. 

\begin{remark}
After obtaining some qualitative properties of solutions to \eqref{up-coupled} and referring those for the wave equation \eqref{us-waves}, we may claim the desired properties for the original thermoealstic system \eqref{Linear-Thermoelastic-Type} in two- and three-dimensions according to the relation \eqref{Helm-solution}. Straightforwardly, due to $\rott u\equiv0$ when $x\in\mb{R}$, our results for the coupled system \eqref{up-coupled} exactly coincide with those for the original system \eqref{Linear-Thermoelastic-Type}  when $n=1$.
\end{remark}

\begin{remark}
If one considers the irrotational thermoelastic system (for instance, \cite[Sections 3 and 4]{Racke=2003}) that is \eqref{Linear-Thermoelastic-Type} carrying $\rott u\equiv0$, then the irrotational system will turn into \eqref{up-coupled}, and all results in this paper immediately work for $u=(u^1,\cdots,u^n)$ with $n=1,2,3$.
\end{remark}

\subsection{Main result on optimal estimates of the potential solution}
Let us state the first result in this paper concerning $L^2$ optimal estimates of each element of the potential solution $u^{p_0}$, particularly, it implies infinite time $L^2$-blowup of solution when $n=1,2$. However, the solution will decay polynomially when $n=3$.
\begin{theorem}\label{Thm-Optimal-Est}
Let us consider the hyperbolic-parabolic coupled system \eqref{up-coupled} for $n=1,2,3$ carrying initial datum $u^{k,p_0}_0,u^{k,p_0}_1,\theta_0\in L^2\cap L^{1}$ with $k=1,\dots,n$. Then, the elastic displacement $u^{k,p_0}$ satisfies the following optimal estimates:
	\begin{align}\label{Estimates:Upper-Bound}
		\ml{A}_n(t)|\mb{A}|\lesssim \|u^{k,p_0}(t,\cdot)\|_{L^2}\lesssim\ml{A}_n(t)\left\|\left(u_0^{k,p_0},u_1^{k,p_0},\theta_0\right)\right\|_{(L^2\cap L^1)^3}
	\end{align}
for $t\gg1$, where the time-dependent coefficient $\ml{A}_n(t)$ was defined in \eqref{Decay-fun}, and the positive constant on the left-hand side is defined by
\begin{align}\label{Non-trivial-01}
\mb{A}^2:=|P_{u_1^{k,p_0}}|^2+ |P_{\theta_0}|^2.
\end{align}
Namely, provided that $|\mb{A}|\neq0$, then the optimal estimates $\|u^{k,p_0}(t,\cdot)\|_{L^2}\simeq\ml{A}_n(t)$ hold for $n=1,2,3$ and any $t\gg1$.
\end{theorem}
\begin{remark}\label{Rem-theta}
According to the relation
\begin{align*}
	\theta(t,x)=-\frac{1}{i\gamma_1}\ml{F}^{-1}_{\xi\to x}\left[\frac{1}{\xi_k}\left(\widehat{u}^{k,p_0}_{tt}(t,\xi)+b^2|\xi|^2\widehat{u}^{k,p_0}(t,\xi)\right)\right],
\end{align*}
one may obtain optimal estimates for the temperature variable without any additional difficulty. Precisely, if $u_0^{k,p_0}\in H^1\cap L^{1}$ and $u_1^{k,p_0},\theta_0\in L^2\cap L^{1}$ equipping $|\mb{A}|\neq0$ defined in \eqref{Non-trivial-01}, then the optimal decay estimates $\|\theta(t,\cdot)\|_{L^2}\simeq t^{-\frac{n}{4}}$ hold for $n=1,2,3$ any any $t\gg1$.
\end{remark}
\begin{remark}\label{Rem-01}
Let us recall the large-time behavior of the free wave model
\begin{align}\label{Wave-Eq}
\begin{cases}
w_{tt}-\Delta w=0,\\
w(0,x)=w_0(x), \ w_t(0,x)=w_1(x),
\end{cases}
\end{align}
with $(t,x)\in\mb{R}_+\times\mb{R}^n$ for $n=1,2$. Under $(L^2\cap L^1)\times (L^2\cap L^{1,1})$ regularity with $|P_{w_1}|\neq0$ assumption on initial datum, the recent work \cite{Ikehata=2022-wave} got the optimal growth estimates
\begin{align}\label{Est-wave}
\|w(t,\cdot)\|_{L^2}\simeq\begin{cases}
\sqrt{t}&\mbox{if}\ \ n=1,\\
\sqrt{\ln t}&\mbox{if}\ \ n=2,
\end{cases}
\end{align}
for $t\gg1$. Note that the weighted $L^1$ regularity for the second data will be relaxed by $L^1$ only in Corollary \ref{Coro-Appendix} with the aid of a different idea. In Theorem \ref{Thm-Optimal-Est}, we indeed obtained the optimal estimates of the thermoelasticity \eqref{up-coupled} (or the wave equation \eqref{Wave-Eq} coupled with Fourier's law of heat conduction) by the following one:
\begin{align*}
\|u^{k,p_0}(t,\cdot)\|_{L^2}\simeq\begin{cases}
\sqrt{t}&\mbox{if}\ \ n=1,\\
\sqrt{\ln t}&\mbox{if}\ \ n=2,\\
t^{-\frac{1}{4}}&\mbox{if}\ \ n=3,
\end{cases}
\end{align*}
for $t\gg1$. We may notice that the growth rates for the free waves and the thermoelastic system are exactly the same if $n=1,2$, but the solution decays polynomially with the aid of thermal dissipation generated by Fourier's law. One may see Table \ref{tab:table1} in detail.
\renewcommand\arraystretch{1.4}
\begin{table}[h!]
	\begin{center}
		\caption{Influence from the wave model and Fourier's law of heat conduction}
		\medskip
		\label{tab:table1}
		\begin{tabular}{cccc} 
			\toprule
			Dimensions & $n=1$ & $n=2$ & $n=3$\\
			\midrule
			Free waves property & $\sqrt{t}$ & $\sqrt{\log t}$ & -- \\
			Heats property (Fourier's law) & $t^{-\frac{1}{4}}$& $t^{-\frac{1}{2}}$& $t^{-\frac{3}{4}}$\\  
			Thermoelastic system property & $\sqrt{t}$ & $\sqrt{\log t}$ & $t^{-\frac{1}{4}}=t^{\frac{1}{2}}\cdot t^{-\frac{3}{4}}$\\
			\hline 
			Crucial influence & Waves & Waves & Waves + Fourier's law\\
			\bottomrule
			\multicolumn{4}{l}{\emph{$*$The terminology ``property'' specializes the time-dependent coefficient in the $L^2$}}\\
						\multicolumn{4}{l}{ \ \ \emph{estimates of the solution.}} 
		\end{tabular}
	\end{center}
\end{table}

\noindent It is worth noting that all large-time properties in Table \ref{tab:table1} are optimal in the sense of same behaviors for upper bound and lower bound of the elastic displacement in the $L^2$ norm. In particular, concerning $n=1,2$, we may observe that the large-time properties of the thermoelastic system \eqref{up-coupled} are not influenced by the heat conduction. Analogously, this phenomenon is also valid for the classical thermoelastic system \eqref{Linear-Thermoelastic-Type} due to the decomposition \eqref{Helm-solution} and the fact that the wave solution $u^{s_0}$ fulfills the estimates \eqref{Est-wave}.
\end{remark}
\begin{remark}
One may find the growth/decay phenomena between the strong damping (or the so-called viscoelastic damping) and thermal dissipation from Fourier's law on the wave model \eqref{Wave-Eq} are the same. Indeed, the authors of \cite{Ikehata=2014,Ikehata-Onodera=2017} showed that the solution of strongly damped waves (i.e. the Cauchy problem for $w_{tt}-\Delta w-\Delta w_t=0$) in the $L^2$ norm satisfies the estimates with $\ml{A}_n(t)$ to be optimal growth or decay rates. Importantly, Theorem \ref{Thm-Optimal-Est} implies the same estimates of the thermoelastic system as those in the strongly damped waves.
\end{remark}

\subsection{Main result on asymptotic profiles of the potential solution}
Let us introduce two crucial components for the leading terms as follows:
\begin{align}\label{Pg1}
	\ml{G}_0(t,x)&:=\ml{F}_{\xi\to x}^{-1}\left(\frac{\sin(\sqrt{b^2+\gamma_1\gamma_2}|\xi|t)}{\sqrt{b^2+\gamma_1\gamma_2}|\xi|}\mathrm{e}^{-\frac{\kappa\gamma_1\gamma_2}{2(b^2+\gamma_1\gamma_2)}|\xi|^2t}\right),\\
	\ml{G}_{1,k}(t,x)&:=\ml{R}_k\ml{F}^{-1}_{\xi\to x}\left[\frac{-i\gamma_1}{(b^2+\gamma_1\gamma_2)|\xi|}\left(\cos\left(\sqrt{b^2+\gamma_1\gamma_2}|\xi|t\right)\mathrm{e}^{-\frac{\kappa\gamma_1\gamma_2}{2(b^2+\gamma_1\gamma_2)}|\xi|^2t}-\mathrm{e}^{-\frac{\kappa b^2}{b^2+\gamma_1\gamma_2}|\xi|^2t}\right)\right].\label{Pg2}
\end{align}
The first function comes from 
\begin{align*}
\mbox{the diffusion-waves:}\quad \ml{F}^{-1}_{\xi\to x}\left(\frac{\sin(\beta_1|\xi|t)}{\beta_1|\xi|}\mathrm{e}^{-\beta_2|\xi|^2t}\right).
\end{align*}
Moreover, the second one may be regarded as a linear combination of
\begin{align*}
\mbox{the diffusion-waves:}\quad \ml{F}^{-1}_{\xi\to x}\left(\cos(\beta_1|\xi|t)\mathrm{e}^{-\beta_2|\xi|^2t}\right)\ \ \mbox{and}\ \ \mbox{the Gaussian kernel:}\quad \ml{F}^{-1}_{\xi\to x}\left(\mathrm{e}^{-\beta_0|\xi|^2t}\right)
\end{align*}
associated with the Riesz transform  $\ml{R}_k$ and the singularity $|\xi|^{-1}$ near $|\xi|=0$, where the multiplier is defined by
\begin{align*}
\widehat{\ml{R}_kf}(\xi):=-\frac{\xi_k}{|\xi|}\widehat{f}(\xi)\ \ \mbox{for}\ \ k=1,\dots,n.
\end{align*}
By taking the leading term
\begin{align*}
\varphi^{k}(t,x):=\ml{G}_0(t,x)P_{u_1^{k,p_0}}+\ml{G}_{1,k}(t,x)P_{\theta_0}\ \ \mbox{for}\ \ k=1,\dots,n,
\end{align*}
we state the optimal estimates for the solution subtracting it in the $L^2$ norm.

\begin{theorem}\label{Thm-Optimal-Lead}
	Let us consider the hyperbolic-parabolic coupled system \eqref{up-coupled} for $n=1,2,3$ carrying initial datum $u^{k,p_0}_0\in L^2\cap L^{1}$ and $u^{k,p_0}_1,\theta_0\in L^2\cap L^{1,1}$ with $k=1,\dots,n$. Then, the elastic displacement $u^{k,p_0}$ satisfies the following optimal refined estimates:
	\begin{align}\label{Estimates:refined-Upper-Bound}
t^{-\frac{n}{4}}|\mb{B}|\lesssim\|u^{k,p_0}(t,\cdot)-\varphi^{k}(t,\cdot)\|_{L^2}\lesssim t^{-\frac{n}{4}}\left\|\left(u^{k,p_0}_0,u^{k,p_0}_1,\theta_0\right)\right\|_{(L^2\cap L^1)\times(L^2\cap L^{1,1})^2}
	\end{align}
for $t\gg1$, where the positive constant on the left-hand side is defined by
\begin{align}\label{Conditioin_B}
\mb{B}^2:=  |P_{u_0^{k,p_0}}|^{2}+  |P_{u_1^{k,p_0}}|^{2} +|P_{\theta_0}|^2  + |M_{u_1^{k,p_0}}|^{2}+ |M_{\theta_{0}}|^{2}.
\end{align}
Namely, provided that $|\mb{B}|\neq0$, then the optimal estimates $\|u^{k,p_0}(t,\cdot)-\varphi^k(t,\cdot)\|_{L^2}\simeq t^{-\frac{n}{4}}$ hold for $n=1,2,3$ and any $t\gg1$.
\end{theorem}
\begin{remark}
With the same reason as Remark \ref{Rem-theta}, by constructing 
\begin{align*}
\widetilde{\varphi}(t,x)=\frac{1}{i\gamma_1}\frac{1}{\ml{R}_k|D|}(\partial_t^2-b^2\Delta)\varphi^k(t,x),
\end{align*}
it is not difficult to get the optimal decay estimates $\|\theta(t,\cdot)-\widetilde{\varphi}(t,\cdot)\|_{L^2}\simeq t^{-\frac{1}{2}-\frac{n}{4}}$ for $n=1,2,3$ and any $t\gg1$, where we assumed $u_0^{k,p_0}\in H^1\cap L^{1}$ and $u_1^{k,p_0},\theta_0\in L^2\cap L^{1,1}$ with $|\mb{B}|\neq0$.
\end{remark}
\begin{remark}
The new condition $|\mb{B}|\neq0$ from \eqref{Conditioin_B} guarantees non-vanishing lower bounds in optimal estimates. This condition still holds even $|\mb{A}|=0$ defined in \eqref{Non-trivial-01}.
\end{remark}
\begin{remark}
	One of our new observations is the optimal leading term $\varphi^k(t,x)$ having diffusion-waves structure rather than parabolic structures for energy terms. It tells us the importance of wave properties in the large-time behaviors of the classical thermoelastic system. The differences between the leading terms of each element in $(u^{1,p_0},\cdots, u^{k,p_0})$ are reflected by two parts: the corresponding initial data $u^{k,p_0}_1$ and the Riesz transform $\ml{R}_k$ in the function $\ml{G}_{1,k}(t,x)$.
\end{remark}
\begin{remark}
In the proof of Theorem \ref{Thm-Optimal-Lead}, we will demonstrate
\begin{align*}
\|\varphi^k(t,\cdot)\|_{L^2}\simeq\ml{A}_n(t)\left(|P_{u_1^{k,p_0}}|+|P_{\theta_0}|\right)
\end{align*}
for $t\gg1$ and $n=1,2,3$. In the view of the optimal estimates in Theorem \ref{Thm-Optimal-Est}, by subtracting the leading term $\varphi^k(t,\cdot)$ in the $L^2$ norm, we arrive at the faster and optimal decay estimate \eqref{Estimates:refined-Upper-Bound}, which hints large-time behaviors since
\begin{align*}
\lim\limits_{t\to\infty}\|u^{k,p_0}(t,\cdot)-\varphi^k(t,\cdot)\|_{L^2}=0.
\end{align*}
 The decay rate has been improved by $t^{-\frac{3}{4}}$ when $n=1$; $(t\ln t)^{-\frac{1}{2}}$ when $n=2$; $t^{-\frac{1}{2}}$ when $n=3$.
\end{remark}

As a by-product of Theorem \ref{Thm-Optimal-Lead}, we may get second-order asymptotic profile for large-time. Let us introduce a function
\begin{align*}
\psi^k(t,x)&:=\nabla\ml{G}_0(t,x)\circ M_{u_1^{k,p_0}}+\nabla\ml{G}_{1,k}(t,x)\circ M_{\theta_0}+\ml{G}_2(t,x)P_{u_0^{k,p_0}}\\
&\quad\ +\big(\ml{H}_0(t,x)+\ml{G}_3(t,x)\big)P_{u_1^{k,p_0}}+\big(\ml{H}_{1,k}(t,x)+\ml{G}_{4,k}(t,x)\big)P_{\theta_0}.
\end{align*}
In the above, the auxiliary functions are defined by
\begin{align*}
\ml{G}_2(t,x)&:=\frac{\gamma_1\gamma_2}{b^2+\gamma_1\gamma_2}\ml{F}^{-1}_{\xi\to x}\left(\mathrm{e}^{-\frac{\kappa b^2}{b^2+\gamma_1\gamma_2}|\xi|^2t}\right)+\frac{b^2}{b^2+\gamma_1\gamma_2}\ml{F}^{-1}_{\xi\to x}\left[\cos\left(\sqrt{b^2+\gamma_1\gamma_2}|\xi|t\right)\mathrm{e}^{-\frac{\kappa\gamma_1\gamma_2}{2(b^2+\gamma_1\gamma_2)}|\xi|^2t}\right],\\ 
\ml{G}_3(t,x)&:=\frac{\kappa\gamma_1\gamma_2}{(b^2+\gamma_1\gamma_2)^2}\ml{F}_{\xi\to x}^{-1}\left(\mathrm{e}^{-\frac{\kappa b^2}{b^2+\gamma_1\gamma_2}|\xi|^2t}-\cos\left(\sqrt{b^2+\gamma_1\gamma_2}|\xi|t\right)\mathrm{e}^{-\frac{\kappa \gamma_1\gamma_2}{2(b^2+\gamma_1\gamma_2)}|\xi|^2t}\right),\\ 
\ml{G}_{4,k}(t,x)&:=-\frac{i\gamma_1\kappa(\gamma_1\gamma_2-2b^2)}{(b^2+\gamma_1\gamma_2)^{5/2}}\ml{R}_k\ml{F}^{-1}_{\xi\to x}\left[\sin\left(\sqrt{b^2+\gamma_1\gamma_2}|\xi|t\right)\mathrm{e}^{-\frac{\kappa\gamma_1\gamma_2}{2(b^2+\gamma_1\gamma_2)}|\xi|^2t}\right],
\end{align*}
and
\begin{align*}
\ml{H}_0(t,x)&:=-\frac{\kappa^2\gamma_1\gamma_2(\gamma_1\gamma_2+4b^2)}{8(b^2+\gamma_1\gamma_2)^3}t\,\ml{F}^{-1}_{\xi\to x}\left[|\xi|^2\cos\left(\sqrt{b^2+\gamma_1\gamma_2}|\xi|t\right)\mathrm{e}^{-\frac{\kappa\gamma_1\gamma_2}{2(b^2+\gamma_1\gamma_2)}|\xi|^2t}\right],\\
\ml{H}_{1,k}(t,x)&:=\frac{i\kappa^2\gamma_1^2\gamma_2(\gamma_1\gamma_2+4b^2)}{8(b^2+\gamma_1\gamma_2)^{7/2}}t\,\ml{F}^{-1}_{\xi\to x}\left[\xi_k|\xi|\sin\left(\sqrt{b^2+\gamma_1\gamma_2}|\xi|t\right)\mathrm{e}^{-\frac{\kappa\gamma_1\gamma_2}{2(b^2+\gamma_1\gamma_2)}|\xi|^2t}\right].
\end{align*}
We underline that all these functions are combined by the diffusion-waves, the heat kernel and the Riesz transform. Then, we may state  a faster decay estimate by subtracting the new profile $\psi^k(t,x)$.
\begin{coro}\label{Coro-Second-order}
	Let us consider the hyperbolic-parabolic coupled system \eqref{up-coupled} for $n=1,2,3$ carrying initial datum $u^{k,p_0}_0\in L^2\cap L^{1}$ and $u^{k,p_0}_1,\theta_0\in L^2\cap L^{1,1}$ with $k=1,\dots,n$. Then, the elastic displacement $u^{k,p_0}$ satisfies the following further refined estimates:
	\begin{align}\label{Second-order-expansion}
		\|u^{k,p_0}(t,\cdot)-\varphi^k(t,\cdot)-\psi^k(t,\cdot)\|_{L^2}=o(t^{-\frac{n}{4}})
	\end{align}
for $t\gg1$, where the right-hand side depends on the norm of initial datum.
\end{coro}
\begin{remark}
Comparing with the optimal estimates \eqref{Estimates:refined-Upper-Bound}, by subtracting the additional function $\psi^k(t,\cdot)$ in the $L^2$ norm, we may obtain faster decay estimates with respect to large-time. In other words, $\psi^k(t,x)$ is the second-order profile of the elastic displacement $u^{k,p_0}(t,x)$.
\end{remark}

\section{Asymptotic behaviors of solutions in the Fourier space}
\subsection{Pretreatment by the reduction procedure}
To investigate some large-time behaviors for \eqref{up-coupled} finely, inspired by the recent paper \cite{Chen-Ikehata=2022-plate}, we may employ the so-called \emph{reduction procedure}. Namely, from our motivation of investigating the potential solution $u^{p_0}$, acting the diffusion operator $\partial_t-\kappa\Delta$ on \eqref{up-coupled}$_1$ and combining the resultant with  \eqref{up-coupled}$_2$, we deduce
\begin{align*}
0&=(\partial_t-\kappa\Delta)(u_{tt}^{p_0}-b^2\Delta u^{p_0})+\gamma_1\nabla(\theta_{t}-\kappa\Delta\theta)\\
&=u_{ttt}^{p_0}-\kappa\Delta u_{tt}^{p_0}-(b^2+\gamma_1\gamma_2)\Delta u^{p_0}_t+\kappa b^2\Delta^2 u^{p_0},
\end{align*}
where we employed $\nabla\divv u^{p_0}=\Delta u^{p_0}$ since $\nabla\times u^{p_0}=0$ for $n=2,3$, and it is trivial for $n=1$. In other words, the coupled system \eqref{up-coupled} can be reduced to the following third-order (in time) PDEs with $n$-scalar equations:
\begin{align}\label{Eq-Fourth-PDEs}
\begin{cases}
u_{ttt}^{p_0}-\kappa\Delta u_{tt}^{p_0}-(b^2+\gamma_1\gamma_2)\Delta u^{p_0}_t+\kappa b^2\Delta^2 u^{p_0}=0,\\
u^{p_0}(0,x)=u_0^{p_0}(x),\ u^{p_0}_t(0,x)=u_1^{p_0}(x),\ u^{p_0}_{tt}(0,x)=b^2\Delta u^{p_0}_0(x)-\gamma_1\nabla\theta_0(x),
\end{cases}
\end{align}
with $(t,x)\in\mb{R}_+\times\mb{R}^n$, where the vector (or scalar in one-dimension) is $u^{p_0}=(u^{1,p_0},\cdots,u^{n,p_0})\in\mb{R}^n$ with $n=1,2,3$. The model \eqref{Eq-Fourth-PDEs} is still a vector equation since the gradient operator in $u_{tt}^{p_0}(0,x)$.

Let us apply the partial Fourier transform with respect to spatial variables to the higher-order evolution model \eqref{Eq-Fourth-PDEs}. It yields
\begin{align}\label{Eq-Fourth-Fourier}
	\begin{cases}
\widehat{u}^{p_0}_{ttt}+\kappa|\xi|^2\widehat{u}^{p_0}_{tt}+(b^2+\gamma_1\gamma_2)|\xi|^2\widehat{u}^{p_0}_t+\kappa b^2|\xi|^4\widehat{u}^{p_0}=0,\\
	\widehat{u}^{p_0}(0,\xi)=\widehat{u}^{p_0}_0(\xi),\ \widehat{u}^{p_0}_t(0,\xi)=\widehat{u}^{p_0}_1(\xi),\ \widehat{u}^{p_0}_{tt}(0,\xi)=-b^2|\xi|^2\widehat{u}^{p_0}_0(\xi)-i\gamma_1\xi\widehat{\theta}_0(\xi),
	\end{cases}
\end{align}
with $(t,\xi)\in\mb{R}_+\times\mb{R}^n$, and
\begin{align*}
	\widehat{\theta}=-\frac{1}{i\gamma_1\xi_k}(\widehat{u}_{tt}^{k,p_0}+b^2|\xi|^2\widehat{u}^{k,p_0})\ \ \mbox{with}\ \ k=1,\dots,n,
\end{align*}
originated from the Fourier transform for \eqref{up-coupled}$_1$. The corresponding characteristic equation to \eqref{Eq-Fourth-Fourier} is given by the $|\xi|$-dependent cubic equation
\begin{align}\label{cubic-eq}
	\lambda^3+\kappa|\xi|^2\lambda^2+(b^2+\gamma_1\gamma_2)|\xi|^2\lambda+\kappa b^2|\xi|^4=0.
\end{align}
Later, without using explicit root's formula to the cubic equation, to facilitate the asymptotic analysis, we will separate the discussion into three parts according to the size of frequencies.
 
 \subsection{Asymptotic expansions for the kernels}
 At the beginning of this subsection, basing on WKB analysis we claim the next expansions, whose proof is straightforward. To be specific, higher-order Taylor-like expansions with respect to $|\xi|$ will be used as $\xi\in\ml{Z}_{\intt}(\varepsilon_0)\cup\ml{Z}_{\extt}(N_0)$ with $\varepsilon_0\ll 1$ as well as $N_0\gg1$, and a contradiction argument associated with continuity of characteristic roots (see, for example, \cite{Jachmann-Reissig=2009}) is valid for $\xi\in\ml{Z}_{\bdd}(\varepsilon_0,N_0)$ since $\Re\lambda_j<0$ when $\xi\in\ml{Z}_{\intt}(\varepsilon_0)\cup\ml{Z}_{\extt}(N_0)$.
 
 \begin{prop}\label{Prop-Expansions}  The characteristic roots $\lambda_j=\lambda_j(|\xi|)$ with $j=1,2,3$ to the $|\xi|$-dependent cubic equation \eqref{cubic-eq} can be expanded by the next way.
 \begin{itemize}
 	\item Concerning $\xi\in\ml{Z}_{\intt}(\varepsilon_0)$, three roots behave as
 	\begin{align*}
 	\lambda_{1}&=-\frac{\kappa b^2}{b^2+\gamma_1\gamma_2}|\xi|^2-\frac{\kappa^3b^4\gamma_1\gamma_2}{(b^2+\gamma_1\gamma_2)^4}|\xi|^4+\ml{O}(|\xi|^6),\\
 	\lambda_{2,3}&=\pm i\sqrt{b^2+\gamma_1\gamma_2}|\xi|-\frac{\kappa\gamma_1\gamma_2}{2(b^2+\gamma_1\gamma_2)}|\xi|^2\mp i\frac{\kappa^2 \gamma_1\gamma_2(\gamma_1\gamma_2+4b^2)}{8(b^2+\gamma_1\gamma_2)^{5/2}}|\xi|^3+\ml{O}(|\xi|^4).
 	\end{align*}
 \item Concerning $\xi\in\ml{Z}_{\extt}(N_0)$, three roots behave as
 \begin{align*}
 \lambda_1&=-\kappa|\xi|^2+\ml{O}(|\xi|),\\
 \lambda_{2,3}&=\pm ib|\xi|-\frac{\gamma_1\gamma_2}{2\kappa}+\ml{O}(|\xi|^{-1}).
 \end{align*}
\item Concerning $\xi\in\ml{Z}_{\bdd}(\varepsilon_0,N_0)$, the roots fulfill $\Re \lambda_j<0$ for any $j=1,2,3$.
 \end{itemize}
 \end{prop}
\begin{remark}
Different from the usual asymptotic expansions for characteristic roots (no matter Taylor-like expansion method \cite{Ide-Haramoto-Kawashima=2008} or diagonalization procedure \cite{Jachmann-Reissig=2009} in the coupled systems), we not only derived pairwise distinct value with negative real parts when $\xi\in\ml{Z}_{\intt}(\varepsilon_0)$, but also found further expansions of characteristic roots containing $|\xi|^3$- and $|\xi|^4$-terms in Proposition \ref{Prop-Expansions}, which will be applied when we study higher-order profiles and optimal leading terms.
\end{remark}
Due to the fact that the discriminant of the cubic \eqref{cubic-eq} is strictly negative for $\xi\in\ml{Z}_{\intt}(\varepsilon_0)$, two complex (non-real) roots $\lambda_{2,3}$ are conjugate, namely, $\lambda_{2,3}=\lambda_{\mathrm{R}}\pm i\lambda_{\mathrm{I}}$ for small frequencies carrying the asymptotic expansions
\begin{align*}
\lambda_{\mathrm{R}}=-\frac{\kappa\gamma_1\gamma_2}{2(b^2+\gamma_1\gamma_2)}|\xi|^2+\ml{O}(|\xi|^4),\ \ \lambda_{\mathrm{I}}=\sqrt{b^2+\gamma_1\gamma_2}|\xi|-\frac{\kappa^2 \gamma_1\gamma_2(\gamma_1\gamma_2+4b^2)}{8(b^2+\gamma_1\gamma_2)^{5/2}}|\xi|^3+\ml{O}(|\xi|^4).
\end{align*}
From the above setting, the solution $\widehat{u}^{k,p_0}=\widehat{u}^{k,p_0}(t,\xi)$ for $\xi\in\ml{Z}_{\intt}(\varepsilon_0)$ owns the representation
\begin{align*}
	\widehat{u}^{k,p_0}&=\frac{-(\lambda_{\mathrm{I}}^2+\lambda_{\mathrm{R}}^2)\widehat{u}_0^{k,p_0}+2\lambda_{\mathrm{R}}\widehat{u}_1^{k,p_0}-\widehat{u}_2^{k,p_0}}{2\lambda_{\mathrm{R}}\lambda_1-\lambda_{\mathrm{I}}^2-\lambda_{\mathrm{R}}^2-\lambda_1^2}\mathrm{e}^{\lambda_1t}+\frac{(2\lambda_{\mathrm{R}}\lambda_1-\lambda_1^2)\widehat{u}_0^{k,p_0}-2\lambda_{\mathrm{R}}\widehat{u}_1^{k,p_0}+\widehat{u}_2^{k,p_0}}{2\lambda_{\mathrm{R}}\lambda_1-\lambda_{\mathrm{I}}^2-\lambda_{\mathrm{R}}^2-\lambda_1^2}\cos(\lambda_{\mathrm{I}}t)\mathrm{e}^{\lambda_{\mathrm{R}}t}\\
	&\quad+\frac{\lambda_1(\lambda_{\mathrm{R}}\lambda_1+\lambda_{\mathrm{I}}^2-\lambda_{\mathrm{R}}^2)\widehat{u}_0^{k,p_0}+(\lambda_{\mathrm{R}}^2-\lambda_{\mathrm{I}}^2-\lambda_1^2)\widehat{u}_1^{k,p_0}-(\lambda_{\mathrm{R}}-\lambda_1)\widehat{u}_2^{k,p_0}}{\lambda_{\mathrm{I}}(2\lambda_{\mathrm{R}}\lambda_1-\lambda_{\mathrm{I}}^2-\lambda_{\mathrm{R}}^2-\lambda_1^2)}\sin(\lambda_{\mathrm{I}}t)\mathrm{e}^{\lambda_{\mathrm{R}}t},
\end{align*}
where the last data is fixed by $\widehat{u}_2^{k,p_0}:=-b^2|\xi|^2\widehat{u}^{k,p_0}_0-i\gamma_1\xi_k\widehat{\theta}_0$ with $k=1,\dots,n$. Let us reformulate the representation with the aid of the last data so that
\begin{align}\label{Rep-01}
	\widehat{u}^{k,p_0}&=\frac{(b^2|\xi|^2-\lambda_{\mathrm{I}}^2-\lambda_{\mathrm{R}}^2)\widehat{u}_0^{k,p_0}+2\lambda_{\mathrm{R}}\widehat{u}_1^{k,p_0}+i\gamma_1\xi_k\widehat{\theta}_0}{2\lambda_{\mathrm{R}}\lambda_1-\lambda_{\mathrm{I}}^2-\lambda_{\mathrm{R}}^2-\lambda_1^2}\mathrm{e}^{\lambda_1t}\notag\\
	&\quad+\frac{(2\lambda_{\mathrm{R}}\lambda_1-\lambda_1^2-b^2|\xi|^2)\widehat{u}_0^{k,p_0}-2\lambda_{\mathrm{R}}\widehat{u}_1^{k,p_0}-i\gamma_1\xi_k\widehat{\theta}_0}{2\lambda_{\mathrm{R}}\lambda_1-\lambda_{\mathrm{I}}^2-\lambda_{\mathrm{R}}^2-\lambda_1^2}\cos(\lambda_{\mathrm{I}}t)\mathrm{e}^{\lambda_{\mathrm{R}}t}\notag\\
	&\quad+\frac{[\lambda_1(\lambda_{\mathrm{R}}\lambda_1+\lambda_{\mathrm{I}}^2-\lambda_{\mathrm{R}}^2)+b^2|\xi|^2(\lambda_{\mathrm{R}}-\lambda_{1})]\widehat{u}_0^{k,p_0}}{\lambda_{\mathrm{I}}(2\lambda_{\mathrm{R}}\lambda_1-\lambda_{\mathrm{I}}^2-\lambda_{\mathrm{R}}^2-\lambda_1^2)}\sin(\lambda_{\mathrm{I}}t)\mathrm{e}^{\lambda_{\mathrm{R}}t}\notag\\
	&\quad+\frac{(\lambda_{\mathrm{R}}^2-\lambda_{\mathrm{I}}^2-\lambda_1^2)\widehat{u}_1^{k,p_0}+i\gamma_1\xi_k(\lambda_{\mathrm{R}}-\lambda_1)\widehat{\theta}_0}{\lambda_{\mathrm{I}}(2\lambda_{\mathrm{R}}\lambda_1-\lambda_{\mathrm{I}}^2-\lambda_{\mathrm{R}}^2-\lambda_1^2)}\sin(\lambda_{\mathrm{I}}t)\mathrm{e}^{\lambda_{\mathrm{R}}t}.
\end{align}
We should underline that the formula \eqref{Rep-01} still holds for $|\xi|\geqslant N_0\gg1$, however, these components will be modified by
\begin{align*}
\lambda_{\mathrm{R}}=-\frac{\gamma_1\gamma_2}{2\kappa}+\ml{O}(|\xi|^{-1}),\ \ \lambda_{\mathrm{I}}=b|\xi|+\ml{O}(|\xi|^{-1}),
\end{align*}
since the strictly negative discriminant for $\xi\in\ml{Z}_{\extt}(N_0)$.
\subsection{Pointwise estimates and auxiliary functions in the Fourier space}
The solution formula \eqref{Rep-01} still seems too complex to analyze its asymptotic behaviors. For this reason, we have to build several bridges by auxiliary functions.
Indeed, as $\xi\in\ml{Z}_{\intt}(\varepsilon_0)$, we extract the dominant terms $\widehat{J}_0^k=\widehat{J}_0^k(t,\xi)$ such that
\begin{align*}
\widehat{J}_0^k=-\frac{\lambda_{\mathrm{I}}\sin(\lambda_{\mathrm{I}}t)\widehat{u}_1^{k,p_0}}{2\lambda_{\mathrm{R}}\lambda_1-\lambda_{\mathrm{I}}^2-\lambda_{\mathrm{R}}^2-\lambda_1^2}\mathrm{e}^{\lambda_{\mathrm{R}}t}+\frac{i\gamma_1\xi_k\widehat{\theta}_0}{2\lambda_{\mathrm{R}}\lambda_1-\lambda_{\mathrm{I}}^2-\lambda_{\mathrm{R}}^2-\lambda_1^2}\left( \mathrm{e}^{\lambda_1t}-\cos(\lambda_{\mathrm{I}}t)\mathrm{e}^{\lambda_{\mathrm{R}}t}\right),
\end{align*}
as well as the further one $\widehat{J}_1^k=\widehat{J}_1^k(t,\xi)$ according to
\begin{align*}
\widehat{J}_1^k&:=\frac{(b^2|\xi|^2-\lambda_{\mathrm{I}}^2)\mathrm{e}^{\lambda_1t}-b^2|\xi|^2\cos(\lambda_{\mathrm{I}}t)\mathrm{e}^{\lambda_{\mathrm{R}}t}}{2\lambda_{\mathrm{R}}\lambda_1-\lambda_{\mathrm{I}}^2-\lambda_{\mathrm{R}}^2-\lambda_1^2}\widehat{u}_0^{k,p_0}+\frac{2\lambda_{\mathrm{R}}\widehat{u}_1^{k,p_0}}{2\lambda_{\mathrm{R}}\lambda_1-\lambda_{\mathrm{I}}^2-\lambda_{\mathrm{R}}^2-\lambda_1^2}\left(\mathrm{e}^{\lambda_1t}-\cos(\lambda_{\mathrm{I}}t)\mathrm{e}^{\lambda_{\mathrm{R}}t}\right)\\
&\ \ \quad+\frac{i\gamma_1\xi_k(\lambda_{\mathrm{R}}-\lambda_1)\widehat{\theta}_0}{\lambda_{\mathrm{I}}(2\lambda_{\mathrm{R}}\lambda_1-\lambda_{\mathrm{I}}^2-\lambda_{\mathrm{R}}^2-\lambda_1^2)}\sin(\lambda_{\mathrm{I}}t)\mathrm{e}^{\lambda_{\mathrm{R}}t}.
\end{align*}

Let us next propose some estimates for the above functions and the refined estimates of solutions by subtracting these functions.
\begin{prop}\label{Prop-Pointwise}
Concerning $\xi\in\ml{Z}_{\intt}(\varepsilon_0)$, the following estimates for the auxiliary functions as well as the error terms hold:
\begin{align}
\chi_{\intt}(\xi)|\widehat{J}_0^k|&\lesssim\chi_{\intt}(\xi)\mathrm{e}^{-c|\xi|^2t}\left(\sqrt{t}+\frac{|\sin(|\xi|t)|}{|\xi|}\right)\left(|\widehat{u}_1^{k,p_0}|+|\widehat{\theta}_0|\right),\label{Pointwise-proof-1}\\
\chi_{\intt}(\xi)|\widehat{u}^{k,p_0}|&\lesssim \chi_{\intt}(\xi)
\mathrm{e}^{-c|\xi|^2t}\left[|\widehat{u}_0^{k,p_0}|+ \left(1+\sqrt{t}+\frac{|\sin(|\xi|t)|}{|\xi|}\right)\left(|\widehat{u}_1^{k,p_0}|+|\widehat{\theta}_0|\right)  \right],\notag\\
\chi_{\intt}(\xi)\left(|\widehat{J}_1^k|+|\widehat{u}^{k,p_0}-\widehat{J}_0^k|\right)&\lesssim\chi_{\intt}(\xi)\mathrm{e}^{-c|\xi|^2t}\left(|\widehat{u}_0^{k,p_0}|+|\widehat{u}_1^{k,p_0}|+|\widehat{\theta}_0|\right),\label{Pointwise-proof-2}\\
\chi_{\intt}(\xi)|\widehat{u}^{k,p_0}-\widehat{J}_0^k-\widehat{J}_1^k|&\lesssim\chi_{\intt}(\xi)|\xi|\mathrm{e}^{-c|\xi|^2t}\left(|\widehat{u}_0^{k,p_0}|+|\widehat{u}_1^{k,p_0}|\right),\notag
\end{align}
for any $k=1,\dots,n$.
\end{prop}
\begin{proof}
By using the next trick:
\begin{align}\label{trick-01}
\mathrm{e}^{\lambda_1t}-\cos(\lambda_{\mathrm{I}}t)\mathrm{e}^{\lambda_{\mathrm{R}}t}&=\left(\mathrm{e}^{\lambda_1t}-\mathrm{e}^{\lambda_{\mathrm{R}}t}\right)+\big(1-\cos(\lambda_{\mathrm{I}}t)\big)\mathrm{e}^{\lambda_{\mathrm{R}}t}\notag\\
&=(\lambda_1-\lambda_{\mathrm{R}})t\mathrm{e}^{\lambda_{\mathrm{R}}t}\int_0^1\mathrm{e}^{(\lambda_1-\lambda_{\mathrm{R}})ts}\mathrm{d}s+2\sin^2\left(\tfrac{1}{2}\lambda_{\mathrm{I}}t\right)\mathrm{e}^{\lambda_{\mathrm{R}}t},
\end{align}
and Proposition \ref{Prop-Expansions} for $\xi\in\ml{Z}_{\intt}(\varepsilon_0)$, we claim
\begin{align*}
\chi_{\intt}(\xi)|\widehat{J}_0^k|\lesssim\chi_{\intt}(\xi)\frac{|\sin(|\xi|t)|}{|\xi|}\mathrm{e}^{-c|\xi|^2t}|\widehat{u}_1^{k,p_0}|+\chi_{\intt}(\xi)\left(|\xi|t+\frac{|\sin(|\xi|t)|^2}{|\xi|}\right)\mathrm{e}^{-c|\xi|^2t}|\widehat{\theta}_0|.
\end{align*}
Because of $|\xi|\sqrt{t}\mathrm{e}^{-c_0|\xi|^2t}\lesssim1$ under $c_0\in(0,c)$ and boundedness of $|\sin(|\xi|t)|$, the desired estimate \eqref{Pointwise-proof-1} can be obtained. Similarly, we may get
\begin{align*}
	\chi_{\intt}(\xi)|\widehat{J}_1^k|\lesssim \chi_{\intt}(\xi)\left(1+|\cos(|\xi|t)|\right)\mathrm{e}^{-c|\xi|^2t}\left(|\widehat{u}_0^{k,p_0}|+|\widehat{u}_1^{k,p_0}|\right)+\chi_{\intt}(\xi)|\sin(|\xi|t)|\mathrm{e}^{-c|\xi|^2t}|\widehat{\theta}_0|,
\end{align*}
which leads to the first part of our estimate \eqref{Pointwise-proof-2}. Additionally, a direct subtraction implies
\begin{align*}
&\chi_{\intt}(\xi)|\widehat{u}^{k,p_0}-\widehat{J}_0^k-\widehat{J}_1^k|\\
&\qquad\leqslant \chi_{\intt}(\xi)\left|\frac{-\lambda_{\mathrm{R}}^2\mathrm{e}^{\lambda_1t}+(2\lambda_{\mathrm{R}}\lambda_1-\lambda_1^2)\cos(\lambda_{\mathrm{I}}t)\mathrm{e}^{\lambda_{\mathrm{R}}t}}{2\lambda_{\mathrm{R}}\lambda_1-\lambda_{\mathrm{I}}^2-\lambda_{\mathrm{R}}^2-\lambda_1^2}\widehat{u}_0^{k,p_0}\right|\notag\\
&\qquad\quad+\chi_{\intt}(\xi)\left|\frac{[\lambda_1(\lambda_{\mathrm{R}}\lambda_1+\lambda_{\mathrm{I}}^2-\lambda_{\mathrm{R}}^2)+b^2|\xi|^2(\lambda_{\mathrm{R}}-\lambda_{1})]\widehat{u}_0^{k,p_0}+(\lambda_{\mathrm{R}}^2-\lambda_1^2)\widehat{u}_1^{k,p_0}}{\lambda_{\mathrm{I}}(2\lambda_{\mathrm{R}}\lambda_1-\lambda_{\mathrm{I}}^2-\lambda_{\mathrm{R}}^2-\lambda_1^2)}\sin(\lambda_{\mathrm{I}}t)\mathrm{e}^{\lambda_{\mathrm{R}}t}\right|\\
&\qquad\lesssim\chi_{\intt}(\xi)\left(|\xi|^2+|\xi|^2|\cos(|\xi|t)|+|\xi|\,|\sin(|\xi|t)|\right)\mathrm{e}^{-c|\xi|^2t}|\widehat{u}^{k,p_0}_0|+\chi_{\intt}(\xi)|\xi|\,|\sin(|\xi|t)|\mathrm{e}^{-c|\xi|^2t}|\widehat{u}^{k,p_0}_1|\\
&\qquad \lesssim\chi_{\intt}(\xi)|\xi|\mathrm{e}^{-c|\xi|^2t}\left(|\widehat{u}_0^{k,p_0}|+|\widehat{u}_1^{k,p_0}|\right).
\end{align*}
To end the proof, we apply the triangle inequality resulting
\begin{align*}
	\chi_{\intt}(\xi)|\widehat{u}^{k,p_0}-\widehat{J}_0^k|&\lesssim \chi_{\intt}(\xi)|\widehat{u}^{k,p_0}-\widehat{J}_0^k-\widehat{J}_1^k|+\chi_{\intt}(\xi)|\widehat{J}_1^k|\\
	&\lesssim \chi_{\intt}(\xi)\mathrm{e}^{-c|\xi|^2t}\left(|\widehat{u}_0^{k,p_0}|+|\widehat{u}_1^{k,p_0}|+|\widehat{\theta}_0|\right)
\end{align*}
as well as
\begin{align*}
\chi_{\intt}(\xi)|\widehat{u}^{k,p_0}|&\lesssim \chi_{\intt}(\xi)|\widehat{u}^{k,p_0}-\widehat{J}_0^k|+\chi_{\intt}(\xi)|\widehat{J}_0^k|\\
&\lesssim \chi_{\intt}(\xi)\mathrm{e}^{-c|\xi|^2t}|\widehat{u}_0^{k,p_0}|+\chi_{\intt}(\xi)\left(1+\sqrt{t}+\frac{|\sin(|\xi|t)|}{|\xi|}\right)\mathrm{e}^{-c|\xi|^2t}\left(|\widehat{u}_1^{k,p_0}|+|\widehat{\theta}_0|\right).
\end{align*}
Hence, we complete the proof of this proposition.
\end{proof}

By ignoring the higher-order terms in the auxiliary function $\widehat{J}_0^k$, we may introduce two Fourier multipliers $\widehat{\ml{G}}_0=\widehat{\ml{G}}_0(t,|\xi|)$ and $\widehat{\ml{G}}_{1,k}=\widehat{\ml{G}}_{1,k}(t,\xi)$ as follows:
\begin{align*}
\widehat{\ml{G}}_0&:=\frac{\sin(\sqrt{b^2+\gamma_1\gamma_2}|\xi|t)}{\sqrt{b^2+\gamma_1\gamma_2}|\xi|}\mathrm{e}^{-\frac{\kappa\gamma_1\gamma_2}{2(b^2+\gamma_1\gamma_2)}|\xi|^2t},\\
\widehat{\ml{G}}_{1,k}&:=\frac{i\gamma_1\xi_k}{(b^2+\gamma_1\gamma_2)|\xi|^2}\left(\cos\left(\sqrt{b^2+\gamma_1\gamma_2}|\xi|t\right)\mathrm{e}^{-\frac{\kappa\gamma_1\gamma_2}{2(b^2+\gamma_1\gamma_2)}|\xi|^2t}-\mathrm{e}^{-\frac{\kappa b^2}{b^2+\gamma_1\gamma_2}|\xi|^2t}\right),
\end{align*}
which are the Fourier transforms of \eqref{Pg1} and \eqref{Pg2}, respectively. Therefore, we can derive some approximations in the sense of additional factors $|\xi|^{s}$ by subtracting some Fourier multipliers in the estimate comparing with the one in \eqref{Pointwise-proof-1}.
\begin{prop}\label{Prop-Improvement-J0}
	Concerning $\xi\in\ml{Z}_{\intt}(\varepsilon_0)$, the following estimates for some approximations hold:
	\begin{align}
	\chi_{\intt}(\xi)\left|\widehat{J}_0^k-\widehat{\ml{G}}_0\widehat{u}_1^{k,p_0}-\widehat{\ml{G}}_{1,k}\widehat{\theta}_0\right|&\lesssim\chi_{\intt}(\xi)\mathrm{e}^{-c|\xi|^2t}\left(|\widehat{u}_1^{k,p_0}|+|\widehat{\theta}_0|\right),\label{Pointwise-proof-3}\\
	\chi_{\intt}(\xi)\left|\widehat{J}_0^k-(\widehat{\ml{G}}_0+\widehat{\ml{H}}_0)\widehat{u}_1^{k,p_0}-(\widehat{\ml{G}}_{1,k}+\widehat{\ml{H}}_{1,k})\widehat{\theta}_0\right|&\lesssim\chi_{\intt}(\xi)|\xi|\mathrm{e}^{-c|\xi|^2t}\left(|\widehat{u}_1^{k,p_0}|+|\widehat{\theta}_0|\right),\notag
	\end{align}
for any $k=1,\dots,n$, where $\widehat{\ml{H}}_0=\widehat{\ml{H}}_0(t,|\xi|)$ and $\widehat{\ml{H}}_{1,k}=\widehat{\ml{H}}_{1,k}(t,\xi)$ are defined by
\begin{align*}
\widehat{\ml{H}}_0&:=-\frac{\kappa^2\gamma_1\gamma_2(\gamma_1\gamma_2+4b^2)}{8(b^2+\gamma_1\gamma_2)^3}|\xi|^2t\cos\left(\sqrt{b^2+\gamma_1\gamma_2}|\xi|t\right)\mathrm{e}^{-\frac{\kappa\gamma_1\gamma_2}{2(b^2+\gamma_1\gamma_2)}|\xi|^2t},\\
\widehat{\ml{H}}_{1,k}&:=i\xi_k\frac{\kappa^2\gamma_1^2\gamma_2(\gamma_1\gamma_2+4b^2)}{8(b^2+\gamma_1\gamma_2)^{7/2}}|\xi|t\sin\left(\sqrt{b^2+\gamma_1\gamma_2}|\xi|t\right)\mathrm{e}^{-\frac{\kappa\gamma_1\gamma_2}{2(b^2+\gamma_1\gamma_2)}|\xi|^2t}.
\end{align*}
\end{prop}
\begin{proof}
According to the representations of these functions, we split our first target into two parts
\begin{align*}
&\widehat{J}_0^k-\widehat{\ml{G}}_0\widehat{u}_1^{k,p_0}-\widehat{\ml{G}}_{1,k}\widehat{\theta}_0\\
&\qquad=\left(\frac{-\lambda_{\mathrm{I}}\sin(\lambda_{\mathrm{I}}t)\mathrm{e}^{\lambda_{\mathrm{R}}t}}{2\lambda_{\mathrm{R}}\lambda_1-\lambda_{\mathrm{I}}^2-\lambda_{\mathrm{R}}^2-\lambda_{1}^2}-\frac{\sin(\sqrt{b^2+\gamma_1\gamma_2}|\xi|t)}{\sqrt{b^2+\gamma_1\gamma_2}|\xi|}\mathrm{e}^{-\frac{\kappa\gamma_1\gamma_2}{2(b^2+\gamma_1\gamma_2)}|\xi|^2t} \right)\widehat{u}_1^{k,p_0}\\
&\qquad\quad+\left(\frac{\mathrm{e}^{\lambda_1t}-\cos(\lambda_{\mathrm{I}}t)\mathrm{e}^{\lambda_{\mathrm{R}}t}}{2\lambda_{\mathrm{R}}\lambda_1-\lambda_{\mathrm{I}}^2-\lambda_{\mathrm{R}}^2-\lambda_{1}^2}-\frac{\cos(\sqrt{b^2+\gamma_1\gamma_2}|\xi|t)\mathrm{e}^{-\frac{\kappa\gamma_1\gamma_2}{2(b^2+\gamma_1\gamma_2)}|\xi|^2t}-\mathrm{e}^{-\frac{\kappa b^2}{b^2+\gamma_1\gamma_2}|\xi|^2t}}{(b^2+\gamma_1\gamma_2)|\xi|^2} \right)i\gamma_1\xi_k\widehat{\theta}_0\\
&\qquad=:\widehat{I}_0\widehat{u}_1^{k,p_0}+\widehat{I}_1i\gamma_1\xi_k\widehat{\theta}_0.
\end{align*}
Let us begin with estimates by the decomposition
\begin{align*}
\widehat{I}_0&=\left(\frac{-\lambda_{\mathrm{I}}}{2\lambda_{\mathrm{R}}\lambda_{1}-\lambda_{\mathrm{I}}^2-\lambda_{\mathrm{R}}^2-\lambda_{1}^2}-\frac{1}{\sqrt{b^2+\gamma_1\gamma_2}|\xi|}\right)\sin(\lambda_{\mathrm{I}}t)\mathrm{e}^{\lambda_{\mathrm{R}}t}+\frac{\sin(\lambda_{\mathrm{I}}t)-\sin(\sqrt{b^2+\gamma_1\gamma_2}|\xi|t)}{\sqrt{b^2+\gamma_1\gamma_2}|\xi|}\mathrm{e}^{\lambda_{\mathrm{R}}t}\\
&\quad+\frac{\sin(\sqrt{b^2+\gamma_1\gamma_2}|\xi|t)}{\sqrt{b^2+\gamma_1\gamma_2}|\xi|}\left(\mathrm{e}^{\lambda_{\mathrm{R}}t}-\mathrm{e}^{-\frac{\kappa\gamma_1\gamma_2}{2(b^2+\gamma_1\gamma_2)}|\xi|^2t}\right)\\
&=:\widehat{I}_{0,1}+\widehat{I}_{0,2}+\widehat{I}_{0,3}.
\end{align*}
The direct computations find $\lambda_{\mathrm{I}}-\sqrt{b^2+\gamma_1\gamma_2}|\xi|=\ml{O}(|\xi|^3)$ benefited from higher-order expansions of characteristic roots, and one deduces
\begin{align*}
\chi_{\intt}(\xi)|\widehat{I}_{0,1}|&\leqslant \chi_{\intt}(\xi)\left|\frac{-\lambda_{\mathrm{I}}\sqrt{b^2+\gamma_1\gamma_2}|\xi|-2\lambda_{\mathrm{R}}\lambda_{1}+\lambda_{\mathrm{I}}^2+\lambda_{\mathrm{R}}^2+\lambda_{1}^2}{(2\lambda_{\mathrm{R}}\lambda_{1}-\lambda_{\mathrm{I}}^2-\lambda_{\mathrm{R}}^2-\lambda_{1}^2)\sqrt{b^2+\gamma_1\gamma_2}|\xi|}\right|\mathrm{e}^{\lambda_{\mathrm{R}}t}\\
&\lesssim\chi_{\intt}(\xi)|\xi|\mathrm{e}^{-c|\xi|^2t}.
\end{align*}
Next, one may employ Taylor's expansion as $\xi\in\ml{Z}_{\intt}(\varepsilon_0)$ to arrive at
\begin{align}\label{Taylor-Expan}
	\sin(\lambda_{\mathrm{I}}t)-\sin\left(\sqrt{b^2+\gamma_1\gamma_2}|\xi|t\right)=-\frac{\kappa^2\gamma_1\gamma_2(\gamma_1\gamma_2+4b^2)}{8(b^2+\gamma_1\gamma_2)^{5/2}}|\xi|^3t\cos\left(\sqrt{b^2+\gamma_1\gamma_2}|\xi|t\right)+\ml{O}(|\xi|^6)t^2,
\end{align}
whose first-order term yields that
\begin{align*}
\chi_{\intt}(\xi)|\widehat{I}_{0,2}|\lesssim\chi_{\intt}(\xi)|\xi|^2t\mathrm{e}^{-c|\xi|^2t}\lesssim\chi_{\intt}(\xi)\mathrm{e}^{-c|\xi|^2t}.
\end{align*}
Again thanks to the higher-order expansion so that $\lambda_{\mathrm{R}}+\frac{\kappa\gamma_1\gamma_2}{2(b^2+\gamma_1\gamma_2)}|\xi|^2=\ml{O}(|\xi|^4)$, the last difference can be controlled by viewing an integral form
\begin{align*}
\chi_{\intt}(\xi)|\widehat{I}_{0,3}|&\lesssim\chi_{\intt}(\xi)\frac{1}{|\xi|}\mathrm{e}^{-\frac{\kappa\gamma_1\gamma_2}{2(b^2+\gamma_1\gamma_2)}|\xi|^2t}\left|\mathrm{e}^{\lambda_{\mathrm{R}}t+\frac{\kappa\gamma_1\gamma_2}{2(b^2+\gamma_1\gamma_2)}|\xi|^2t}-1\right|\\
&\lesssim\chi_{\intt}(\xi)|\xi|^3t\mathrm{e}^{-\frac{\kappa\gamma_1\gamma_2}{2(b^2+\gamma_1\gamma_2)}|\xi|^2t}\left|\int_0^1\mathrm{e}^{\ml{O}(|\xi|^4)ts}\mathrm{d}s\right|\\
&\lesssim\chi_{\intt}(\xi)|\xi|\mathrm{e}^{-c|\xi|^2t}.
\end{align*}
Summarizing the obtained estimates, we have
\begin{align*}
\chi_{\intt}(\xi)|\widehat{I}_0\widehat{u}_1^{k,p_0}|\lesssim\chi_{\intt}(\xi)\mathrm{e}^{-c|\xi|^2t}|\widehat{u}_1^{k,p_0}|.
\end{align*}
Among them, the worst term is $\widehat{I}_{0,2}$ since the lack of the factor $|\xi|$. If we subtract the additional profile $\widehat{\ml{H}}_0$, it leads to
\begin{align*}
&\chi_{\intt}(\xi)|\widehat{I}_{0,2}-\widehat{\ml{H}}_0|\\
&\qquad\leqslant\chi_{\intt}(\xi)\left|\frac{\sin(\lambda_{\mathrm{I}}t)-\sin(\sqrt{b^2+\gamma_1\gamma_2}|\xi|t)+\frac{\kappa^2\gamma_1\gamma_2(\gamma_1\gamma_2+4b^2)}{8(b^2+\gamma_1\gamma_2)^{5/2}}|\xi|^3t\cos(\sqrt{b^2+\gamma_1\gamma_2}|\xi|t)}{\sqrt{b^2+\gamma_1\gamma_2}|\xi|}\right|\mathrm{e}^{\lambda_{\mathrm{R}}t}\\
&\qquad\quad+\chi_{\intt}(\xi)\frac{\kappa^2\gamma_1\gamma_2(\gamma_1\gamma_2+4b^2)}{8(b^2+\gamma_1\gamma_2)^3}|\xi|^2t\left|\cos\left(\sqrt{b^2+\gamma_1\gamma_2}|\xi|t\right)\right|\left|\mathrm{e}^{\lambda_{\mathrm{R}}t}-\mathrm{e}^{-\frac{\kappa\gamma_1\gamma_2}{2(b^2+\gamma_1\gamma_2)}|\xi|^2t}\right|\\
&\qquad\lesssim\chi_{\intt}(\xi)|\xi|^5t^2\mathrm{e}^{-c|\xi|^2t}\\
&\qquad\lesssim\chi_{\intt}(\xi)|\xi|\mathrm{e}^{-c|\xi|^2t},
\end{align*}
where we employed the expansion \eqref{Taylor-Expan} by moving the first term from the right side to the left one.

Then, we repeat the analogous idea as the previous one to have
\begin{align*}
	\chi_{\intt}(\xi)|\widehat{I}_1|\lesssim\chi_{\intt}(\xi)\left(1+|\xi|t+|\xi|^2t\right)\mathrm{e}^{-c|\xi|^2t}\lesssim\chi_{\intt}(\xi)(1+\sqrt{t})\mathrm{e}^{-c|\xi|^2t},
\end{align*}
where the worst term $|\xi|t$ comes from the mean value theorem
\begin{align*}
\chi_{\intt}(\xi)\left|\cos(\lambda_{\mathrm{I}}t)-\cos\left(\sqrt{b^2+\gamma_1\gamma_2}|\xi|t\right)\right|\lesssim\chi_{\intt}(\xi)|\xi|^3t.
\end{align*}
Again by noticing Taylor's formula
\begin{align*}
\cos(\lambda_{\mathrm{I}}t)-\cos\left(\sqrt{b^2+\gamma_1\gamma_2}|\xi|t\right)=\sin\left(\sqrt{b^2+\gamma_1\gamma_2}|\xi|t\right)\frac{\kappa^2\gamma_1\gamma_2(\gamma_1\gamma_2+4b^2)}{8(b^2+\gamma_1\gamma_2)^{5/2}}|\xi|^3t+\ml{O}(|\xi|^6)t^2,
\end{align*}
we may obtain immediately
\begin{align*}
\chi_{\intt}(\xi)\left|\widehat{I}_1i\gamma_1\xi_k-\widehat{\ml{H}}_{1,k}\right|\lesssim\chi_{\intt}(\xi)|\xi|\mathrm{e}^{-c|\xi|^2t}.
\end{align*}
In conclusion, it yields
\begin{align*}
\chi_{\intt}(\xi)\left|\widehat{J}_0^k-\widehat{\ml{G}}_0\widehat{u}_1^{k,p_0}-\widehat{\ml{G}}_{1,k}\widehat{\theta}_0\right|&\lesssim\chi_{\intt}(\xi)\left(|\widehat{I}_0|\,|\widehat{u}_1^{k,p_0}|+|\widehat{I}_{1}\xi_k|\,|\widehat{\theta}_0|\right)\\
&\lesssim\chi_{\intt}(\xi)\mathrm{e}^{-c|\xi|^2t}\left(|\widehat{u}_1^{k,p_0}|+|\widehat{\theta}_0|\right)
\end{align*}
and
\begin{align*}
	&\chi_{\intt}(\xi)\left|\widehat{J}_0^k-(\widehat{\ml{G}}_0+\widehat{\ml{H}}_0)\widehat{u}_1^{k,p_0}-(\widehat{\ml{G}}_{1,k}+\widehat{\ml{H}}_{1,k})\widehat{\theta}_0\right|\\
	&\qquad\lesssim\chi_{\intt}(\xi)\left(\,\left|\widehat{I}_0-\widehat{\ml{H}}_0\right||\widehat{u}_1^{k,p_0}|+\left|\widehat{I}_1i\gamma_1\xi_k-\widehat{\ml{H}}_{1,k}\right||\widehat{\theta}_0|\right)\\
	&\qquad\lesssim\chi_{\intt}(\xi)|\xi|\mathrm{e}^{-c|\xi|^2t}\left(|\widehat{u}_1^{k,p_0}|+|\widehat{\theta}_0|\right).
\end{align*}
The proof is complete now.
\end{proof}

Secondly, we employ the analogous ideas as the treatment of $\widehat{J}_0^k$ and denote $\widehat{\ml{G}}_{2}=\widehat{\ml{G}}_{2}(t,|\xi|)$, $\widehat{\ml{G}}_{3}=\widehat{\ml{G}}_{3}(t,|\xi|)$  and $\widehat{\ml{G}}_{4,k}=\widehat{\ml{G}}_{4,k}(t,\xi)$ by
\begin{align*}
	\widehat{\ml{G}}_2&:=\frac{\gamma_1\gamma_2}{b^2+\gamma_1\gamma_2}\mathrm{e}^{-\frac{\kappa b^2}{b^2+\gamma_1\gamma_2}|\xi|^2t}+\frac{b^2}{b^2+\gamma_1\gamma_2}\cos\left(\sqrt{b^2+\gamma_1\gamma_2}|\xi|t\right)\mathrm{e}^{-\frac{\kappa\gamma_1\gamma_2}{2(b^2+\gamma_1\gamma_2)}|\xi|^2t},\\ 
	\widehat{\ml{G}}_3&:=\frac{\kappa\gamma_1\gamma_2}{(b^2+\gamma_1\gamma_2)^2}\mathrm{e}^{-\frac{\kappa b^2}{b^2+\gamma_1\gamma_2}|\xi|^2t}-\cos\left(\sqrt{b^2+\gamma_1\gamma_2}|\xi|t\right)\mathrm{e}^{-\frac{\kappa \gamma_1\gamma_2}{2(b^2+\gamma_1\gamma_2)}|\xi|^2t},\\ 
	\widehat{\ml{G}}_{4,k}&:=\frac{i\gamma_1\kappa(\gamma_1\gamma_2-2b^2)}{(b^2+\gamma_1\gamma_2)^{5/2}}\frac{\xi_k}{|\xi|}\sin\left(\sqrt{b^2+\gamma_1\gamma_2}|\xi|t\right)\mathrm{e}^{-\frac{\kappa\gamma_1\gamma_2}{2(b^2+\gamma_1\gamma_2)}|\xi|^2t}.
\end{align*}
Following the proof of the first estimate in Proposition \ref{Prop-Improvement-J0}, we can obtain the next result. Since the method of the demonstration is the same with slight changes of procedure, we omit the proof.
\begin{prop}\label{Prop-Improvement-J1}
	Concerning $\xi\in\ml{Z}_{\intt}(\varepsilon_0)$, the following estimate for an approximation holds:
	\begin{align*}
		\chi_{\intt}(\xi)\left|\widehat{J}_1^k-\widehat{\ml{G}}_2\widehat{u}_0^{k,p_0}-\widehat{\ml{G}}_{3}\widehat{u}_1^{k,p_0}-\widehat{\ml{G}}_{4,k}\widehat{\theta}_0\right|&\lesssim\chi_{\intt}(\xi)|\xi|\mathrm{e}^{-c|\xi|^2t}\left(|\widehat{u}_0^{k,p_0}|+|\widehat{u}_1^{k,p_0}|+|\widehat{\theta}_0|\right)
	\end{align*}
for any $k=1,\dots,n$.
\end{prop}

Eventually, to finish this subsection, we propose pointwise estimates localizing in bounded and large frequencies zones. They will not influence on large-time behaviors since exponential decays.
\begin{prop}\label{Prop-Pointwise-large}
Concerning $\xi\in\ml{Z}_{\bdd}(\varepsilon_0,N_0)\cup\ml{Z}_{\extt}(N_0)$, the following pointwise estimate holds:
\begin{align*}
	\big(\chi_{\bdd}(\xi)+\chi_{\extt}(\xi)\big)|\widehat{u}^{k,p_0}|\lesssim\big(\chi_{\bdd}(\xi)+\chi_{\extt}(\xi)\big)\mathrm{e}^{-ct}\left(|\widehat{u}_0^{k,p_0}|+\frac{1}{\langle\xi\rangle}|\widehat{u}_1^{k,p_0}|+\frac{1}{\langle \xi\rangle^2}|\widehat{\theta}_0|\right)
\end{align*}
for any $k=1,\dots,n$.
\end{prop}
\begin{proof}
Taking consideration of Proposition \ref{Prop-Expansions} as $\xi\in\ml{Z}_{\extt}(N_0)$ in the representation \eqref{Rep-01}, we are able to estimate
\begin{align*}
\chi_{\extt}(\xi)|\widehat{u}^{k,p_0}|
&\lesssim\chi_{\extt}(\xi)\left(\frac{1}{|\xi|^2}\mathrm{e}^{-c|\xi|^2t}+\mathrm{e}^{-ct}\right)|\widehat{u}_0^{k,p_0}|+\chi_{\extt}(\xi)\left(\frac{1}{|\xi|^4}\mathrm{e}^{-c|\xi|^2t}+\frac{1}{|\xi|}\mathrm{e}^{-ct}\right)|\widehat{u}_1^{k,p_0}|\\
&\quad+\chi_{\extt}(\xi)\left(\frac{1}{|\xi|^3}\mathrm{e}^{-c|\xi|^2t}+\frac{1}{|\xi|^2}\mathrm{e}^{-ct}\right)|\widehat{\theta}_0|\\
&\lesssim\chi_{\extt}(\xi)\mathrm{e}^{-ct}\left(|\widehat{u}_0^{k,p_0}|+\frac{1}{|\xi|}|\widehat{u}_1^{k,p_0}|+\frac{1}{|\xi|^2}|\widehat{\theta}_0|\right).
\end{align*}
For another, the last conclusion in Proposition \ref{Prop-Expansions} tells us that the solution in the Fourier space decays exponentially such that
\begin{align*}
\chi_{\bdd}(\xi)|\widehat{u}^{k,p_0}|\lesssim \chi_{\bdd}(\xi) \mathrm{e}^{-ct}\left(|\widehat{u}_0^{k,p_0}|+|\widehat{u}_1^{k,p_0}|+|\widehat{\theta}_0|\right).
\end{align*}
The combination of previous two estimates and $|\xi|\simeq\langle \xi\rangle$ for any $|\xi|\geqslant\varepsilon_0$ completes the proof.
\end{proof}

\section{Large-time asymptotic profiles for the potential solution}
The schedule of this section is arranged by: in Subsection \ref{Sub-Sharp-Est}, we will derive upper bound estimates  and optimal lower bound estimates \eqref{Estimates:Upper-Bound} for any physical dimensions $n=1,2,3$; in Subsection \ref{Sub-Second-Profile}, some estimates for the second-order profiles will be deduced; and finally in Subsection \ref{Sub-Leading}, the optimal leading terms will be investigated.

\subsection{Optimal estimates and first-order profiles for the potential solution}\label{Sub-Sharp-Est}
This part contributes to the proof of Theorem \ref{Thm-Optimal-Est}. The second estimate in Proposition \ref{Prop-Pointwise} shows
\begin{align*}
\|\chi_{\intt}(\xi)\widehat{u}^{k,p_0}(t,\xi)\|_{L^2}&\lesssim\left\|\chi_{\intt}(\xi)\mathrm{e}^{-c|\xi|^2t}|\widehat{u}^{k,p_0}_0(\xi)|\right\|_{L^2}+(1+t)^{\frac{1}{2}}\left\|\chi_{\intt}(\xi)\mathrm{e}^{-c|\xi|^2t}\left(|\widehat{u}^{k,p_0}_1(\xi)|+|\widehat{\theta}_0(\xi)|\right)\right\|_{L^2}\\
&\quad+\left\|\chi_{\intt}(\xi)\frac{|\sin(|\xi|t)|}{|\xi|}\mathrm{e}^{-c|\xi|^2t}\left(|\widehat{u}^{k,p_0}_1(\xi)|+|\widehat{\theta}_0(\xi)|\right)\right\|_{L^2}
\end{align*}
for $n=1,2,3$. According to the inspiring works \cite{Ikehata=2014,Ikehata-Onodera=2017}, we know the sharp estimates in the sense of same behaviors for upper bounds and lower bounds as follows:
\begin{align}
\left\|\mathrm{e}^{-c|\xi|^2t}\right\|_{L^2}&\simeq t^{-\frac{n}{4}},\label{optimal-fact-01}\\
\left\|\frac{|\sin(|\xi|t)|}{|\xi|}\mathrm{e}^{-c|\xi|^2t}\right\|_{L^2}&\simeq \ml{A}_n(t),\label{optimal-fact-02}
\end{align}
for $n=1,2,3$ and any $t\gg1$, where the time-dependent function $\ml{A}_n(t)$ was introduced in \eqref{Decay-fun}. Then, with the aid of H\"older's inequality and the Hausdorff-Young inequality, for large-time $t\gg1$ we arrive at
\begin{align*}
\|\chi_{\intt}(\xi)\widehat{u}^{k,p_0}(t,\xi)\|_{L^2}\lesssim t^{-\frac{n}{4}}\|u_0^{k,p_0}\|_{L^1}+\ml{A}_n(t)\|(u_1^{k,p_0},\theta_0)\|_{(L^1)^2}.
\end{align*}
Moreover, exponential decay estimates in Proposition \ref{Prop-Pointwise-large} imply
\begin{align*}
\left\|\big(\chi_{\bdd}(\xi)+\chi_{\extt}(\xi)\big)\widehat{u}^{k,p_0}(t,\xi)\right\|_{L^2}\lesssim\mathrm{e}^{-ct}\left\|\left(u^{k,p_0}_0,u^{k,p_0}_1,\theta_0\right)\right\|_{(L^2)^3}.
\end{align*}
Finally, we apply the Plancherel theorem and $t^{-\frac{n}{4}}\lesssim\ml{A}_n(t)$ for $n=1,2,3$ to complete the derivation of upper bound estimates in \eqref{Estimates:Upper-Bound}.

By the same way as preceding parts of the text and the uses of \eqref{Pointwise-proof-2} and \eqref{Pointwise-proof-3}, one finds
\begin{align}
	&\left\|\chi_{\intt}(D)\left(u^{k,p_0}(t,\cdot)-\ml{G}_0(t,|D|)u_1^{k,p_0}(\cdot)-\ml{G}_{1,k}(t,D)\theta_0(\cdot)\right)\right\|_{L^2}\notag\\
	&\quad\leqslant\left\|\chi_{\intt}(\xi)\left(\widehat{u}^{k,p_0}(t,\xi)-\widehat{J}_0^k(t,\xi)\right)\right\|_{L^2}+\left\|\chi_{\intt}(\xi)\left(\widehat{J}_0^k(t,\xi)-\widehat{\ml{G}}_0(t,|\xi|)\widehat{u}_1^{k,p_0}(\xi)-\widehat{\ml{G}}_{1,k}(t,\xi)\widehat{\theta}_0(\xi)\right)\right\|_{L^2}\notag\\
	&\quad\lesssim\left\|\chi_{\intt}(\xi)\mathrm{e}^{-c|\xi|^2t}\left(|\widehat{u}^{k,p_0}_0(\xi)|+|\widehat{u}^{k,p_0}_1(\xi)|+|\widehat{\theta}_0(\xi)|\right)\right\|_{(L^2)^3}\notag\\
	&\quad\lesssim t^{-\frac{n}{4}}\left\|\left(u^{k,p_0}_0,u^{k,p_0}_1,\theta_0\right)\right\|_{(L^1)^3} \label{eq:28}
\end{align}
for $t\gg1$. Further applications of the triangle inequality and \eqref{optimal-fact-01} indicate
\begin{align}
	&\left\|\chi_{\intt}(D)\left(u^{k,p_0}(t,\cdot)-\ml{G}_0(t,\cdot)P_{u_1^{k,p_0}}-\ml{G}_{1,k}(t,\cdot)P_{\theta_0}\right)\right\|_{L^2}\notag\\
	&\qquad\lesssim t^{-\frac{n}{4}}\left\|\left(u^{k,p_0}_0,u^{k,p_0}_1,\theta_0\right)\right\|_{(L^1)^3}+\left\|\widehat{\ml{G}}_0(t,|\xi|)\left(\widehat{u}_1^{k,p_0}(\xi)-P_{u_1^{k,p_0}}\right)\right\|_{L^2}+\left\|\widehat{\ml{G}}_{1,k}(t,\xi)\left(\widehat{\theta}_0(\xi)-P_{\theta_0}\right)\right\|_{L^2}\notag\\
	&\qquad\lesssim t^{-\frac{n}{4}}\left\|\left(u^{k,p_0}_0,u^{k,p_0}_1,\theta_0\right)\right\|_{(L^1)^3}+\left\||\xi|\widehat{\ml{G}}_0(t,|\xi|)\right\|_{L^2}\|u_1^{k,p_0}\|_{L^{1,1}}+\left\||\xi|\widehat{\ml{G}}_{1,k}(t,\xi)\right\|_{L^2}\|\theta_0\|_{L^{1,1}}\notag\\
		&\qquad\lesssim t^{-\frac{n}{4}}\left\|\left(u^{k,p_0}_0,u^{k,p_0}_1,\theta_0\right)\right\|_{L^1\times(L^{1,1})^2},\label{Sub-tract-01}
\end{align}
where we have employed the next estimate (see, for instance, \cite[Lemma 2.2]{Ikehata=2014}) in the second line of the last chain:
\begin{align*}
|\widehat{f}(\xi)-P_f|\lesssim|\xi|\,\|f\|_{L^{1,1}}.
\end{align*}
Concerning the other frequencies, we still can get exponential decay estimates
\begin{align}\label{Sub-tract-large-01}
&\left\|\big(\chi_{\bdd}(D)+\chi_{\extt}(D)\big)\left(u^{k,p_0}(t,\cdot)-\ml{G}_0(t,\cdot)P_{u_1^{k,p_0}}-\ml{G}_{1,k}(t,\cdot)P_{\theta_0}\right)\right\|_{L^2}\notag\\
&\qquad\lesssim \mathrm{e}^{-ct}\left\|\left(u^{k,p_0}_0,u^{k,p_0}_1,\theta_0\right)\right\|_{(L^2)^3}+\mathrm{e}^{-ct}\left(|P_{u_1^{k,p_0}}|+|P_{\theta_0}|\right)
\end{align}
for any $t\gg1$ due to $|\xi|\geqslant \varepsilon_0$. Observing the trivial fact $|P_f|\leqslant \|f\|_{L^{1,1}}$, we combine the obtained estimates in the above to complete the proof of \eqref{Estimates:refined-Upper-Bound}.

Let us turn to lower bound estimates for the potential solution $u^{p_0}$. The crucial step is to estimate the Fourier multipliers $\widehat{\ml{G}}_0(t,|\xi|)$ and $\widehat{\ml{G}}_{1,k}(t,\xi)$ from the below when $n=1,2,3$. A direct consequence of \eqref{optimal-fact-02} is
\begin{align}\label{lower-bound-01}
\|\chi_{\intt}(\xi)\widehat{\ml{G}}_0(t,|\xi|)\|_{L^2}\gtrsim \ml{A}_n(t)
\end{align}
for $n=1,2,3$ and any $t\gg1$. Due to the situation that $\widehat{\ml{G}}_{1,k}(t,\xi)$ is not radial symmetric with respect to $\xi$, we will frequently use the following chain for treating the multipliers $\ml{M}(|\xi|)$ containing Riesz transform:
\begin{align}\label{Over-come-Riesz}
	\|\widehat{\ml{R}_k\ml{M}}(|\xi|)\|_{L^2}^2=\int_{\mb{R}^n}|\widehat{\ml{M}}(|\xi|)|^2\frac{\xi_k^2}{|\xi|^2}\mathrm{d}\xi&=\int_0^{\infty}|\widehat{\ml{M}}(r)|^2r^{n-1}\mathrm{d}r\int_{\mb{S}^{n-1}}\omega_k^2\mathrm{d}\sigma_{\omega}\notag\\
	&=\frac{|\mb{S}^{n-1}|}{n}\int_0^{\infty}|\widehat{\ml{M}}(r)|^2r^{n-1}\mathrm{d}r,
\end{align}
where we used polar coordinates with the $(n-1)$-dimensional measure of the unit sphere $\mb{S}^{n-1}$. As a consequence, we know from \eqref{Over-come-Riesz} that
\begin{align}\label{C-01}
\left\|\frac{1}{i}\widehat{\ml{G}}_{1,k}(t,\xi)\right\|_{L^2}\gtrsim \left\|\frac{1}{|\xi|}\left(\mathrm{e}^{-\frac{\kappa b^2}{b^2+\gamma_1\gamma_2}|\xi|^2t}-\cos\left(\sqrt{b^2+\gamma_1\gamma_2}|\xi|t\right)\mathrm{e}^{-\frac{\kappa\gamma_1\gamma_2}{2(b^2+\gamma_1\gamma_2)}|\xi|^2t}\right)\right\|_{L^2},
\end{align}
since the multiplier $\frac{1}{i}\widehat{\ml{G}}_{1,k}(t,\xi)$ is a real function. Indeed, due to the singularity $|\xi|^{-1}$ for $|\xi|\to 0$, the lower bound estimates for it are quite delicate depending on dimensions.
 \begin{prop}\label{Prop-Fourier-mul-Takeda}
 	Let us take any $\beta_j>0$ for $j=0,1,2$. Concerning the Fourier multiplier
 	\begin{align*}
 	\ml{M}_n(t,|\xi|):=\frac{1}{|\xi|}\left(\mathrm{e}^{-\beta_0|\xi|^2t}-\cos(\beta_1|\xi|t)\mathrm{e}^{-\beta_2|\xi|^2t}\right),
 	\end{align*}
 the following optimal estimates hold:
 \begin{align}\label{Upp-1}
 	\|\ml{M}_n(t,|\xi|)\|_{L^2}\simeq\ml{A}_n(t)
 \end{align}
for any $n=1,2,3$ and $t\gg1$.
 \end{prop}
\begin{proof}
	Thanks to exponential decay estimates when $|\xi|\geqslant\varepsilon_0$, the upper bound estimates \eqref{Upp-1} have been finished actually at the beginning of this subsection due to the trick \eqref{trick-01}, boundedness of since functions and \eqref{optimal-fact-02}. For these reasons, we just need to concentrate on \eqref{Upp-1} from the below side, which is the most challenging part in the optimal estimates. To do so, we will separate our discussion into three parts in terms of dimensions.
	\medskip
	
	\noindent\underline{Lower-dimensional case: $n=1$.} With the same philosophy as those in \eqref{trick-01}, in general we may rewrite the multiplier by
	\begin{align*}
		\ml{M}_n(t,|\xi|)=\frac{2}{|\xi|}\left|\sin\left(\tfrac{\beta_1}{2}|\xi|t\right)\right|^2\mathrm{e}^{-\beta_2|\xi|^2t}+(\beta_2-\beta_0)|\xi|t\mathrm{e}^{-\beta_2|\xi|^2t}\int_0^1\mathrm{e}^{(\beta_2-\beta_0)|\xi|^2ts}\mathrm{d}s.
	\end{align*}
Note that the last term in the above will vanish when $\beta_0=\beta_2$. Hence, we apply $2|f-g|^2\geqslant |f|^2-2|g|^2$ to deduce
\begin{align*}
\|\ml{M}_n(t,|\xi|)\|_{L^2}^2&\geqslant\|\chi_{\intt}(\xi)\ml{M}_n(t,|\xi|)\|_{L^2}^2\\
&\geqslant\int_{|\xi|\leqslant\varepsilon_0}\frac{2}{|\xi|^2}\left|\sin\left(\tfrac{\beta_1}{2}|\xi|t\right)\right|^4\mathrm{e}^{-2\beta_2|\xi|^2t}\mathrm{d}\xi-Ct^2\int_{|\xi|\leqslant\varepsilon_0}|\xi|^2\mathrm{e}^{-c|\xi|^2t}\mathrm{d}\xi.
\end{align*}
Let us choose a constant $\alpha_0$ such that $0<\alpha_0<\pi/\beta_1$. That is to say
\begin{align*}
\left|\sin\left(\tfrac{\beta_1}{2}|\xi|t\right)\right|\geqslant C>0\ \ \mbox{for any}\ \ |\xi|\in[\alpha_0t^{-1},2\alpha_0t^{-1}].
\end{align*}
We take large-time $t\gg1$ so that $2\alpha_0t^{-1}<\varepsilon_0$ always holds. Thus, we shrank the domain of the first integral only and derive
\begin{align*}
\|\ml{M}_n(t,|\xi|)\|_{L^2}^2&\geqslant\int_{\alpha_0t^{-1}\leqslant|\xi|\leqslant2\alpha_0t^{-1}}\frac{2}{|\xi|^2}\left|\sin\left(\tfrac{\beta_1}{2}|\xi|t\right)\right|^4\mathrm{e}^{-2\beta_2|\xi|^2t}\mathrm{d}\xi-Ct\int_{|\xi|\leqslant\varepsilon_0}\mathrm{e}^{-c|\xi|^2t}\mathrm{d}\xi\\
&\gtrsim t^2\int_{\alpha_0 t^{-1}\leqslant|\xi|\leqslant2\alpha_0t^{-1}}\mathrm{e}^{-2\beta_2|\xi|^2t}\mathrm{d}\xi-Ct^{1-\frac{n}{2}}\\
&\gtrsim t^{2-n}\mathrm{e}^{-8\alpha_0^2\beta_2t^{-1}}-t^{1-\frac{n}{2}}\\
&\gtrsim t^{2-n}
\end{align*}
for $n=1,2$ and $t\gg1$, where we used \eqref{optimal-fact-01}. Nevertheless, considering $n=2$, the last lower bound is just a constant which seems to be not sharp. For this reason, we will employ another idea in such critical-dimension.\medskip
	
	\noindent\underline{Critical-dimensional case: $n=2$.} By considering polar coordinates and shrinking the resultant domain into $[t^{-\frac{1}{2}},1]$ for $t\gg1$, one may see
	\begin{align}\label{Take-1}
	&\|\ml{M}_2(t,|\xi|)\|_{L^2}^2\gtrsim\int_0^{\infty}\left(\mathrm{e}^{-\beta_0r^2t}-\cos(\beta_1rt)\mathrm{e}^{-\beta_2r^2t}\right)^2 r^{-1}\mathrm{d}r\notag\\
	&\quad \gtrsim\int_{t^{-\frac{1}{2}}}^1\left(\mathrm{e}^{-2\beta_0\sigma^2}-2\cos(\beta_1\sqrt{t}\sigma)\mathrm{e}^{-(\beta_0+\beta_2)\sigma^2}+|\cos(\beta_1\sqrt{t}\sigma)|^2\mathrm{e}^{-2\beta_2\sigma^2}\right)\sigma^{-1}\mathrm{d}\sigma\notag\\
	&\quad \gtrsim\int_{t^{-\frac{1}{2}}}^1\left[\left(\mathrm{e}^{-2\beta_0\sigma^2}+\frac{1}{2}\mathrm{e}^{-2\beta_2\sigma^2}\right)-2\cos(\beta_1\sqrt{t}\sigma)\mathrm{e}^{-(\beta_0+\beta_2)\sigma^2}+\frac{1}{2}\cos(2\beta_1\sqrt{t}\sigma)\mathrm{e}^{-2\beta_2\sigma^2}\right]\sigma^{-1}\mathrm{d}\sigma,
	\end{align}
where we used $2\cos^2z=1+\cos(2z)$ and $\sigma=\sqrt{t}r$ to be a new ansatz. Indeed, we notice
\begin{align*}
\int_{t^{-\frac{1}{2}}}^1\left(\mathrm{e}^{-2\beta_0\sigma^2}+\frac{1}{2}\mathrm{e}^{-2\beta_2\sigma^2}\right)\sigma^{-1}\mathrm{d}\sigma\gtrsim\int_{t^{-\frac{1}{2}}}^1\sigma^{-1}\mathrm{d}\sigma=\frac{1}{2}\ln t
\end{align*}
for $t\gg1$. For another, the application of integration by parts hints
\begin{align*}
-2\int_{t^{-\frac{1}{2}}}^1\frac{\cos(\beta_1\sqrt{t}\sigma)}{\sigma}\mathrm{e}^{-(\beta_0+\beta_2)\sigma^2}\mathrm{d}\sigma&=-\frac{2}{\beta_1\sqrt{t}}\left(\frac{\sin(\beta_1\sqrt{t}\sigma)}{\sigma}\mathrm{e}^{-(\beta_0+\beta_2)\sigma^2}\right)\Big|_{\sigma=t^{-\frac{1}{2}}}^{\sigma=1}\\
&\quad+\frac{2}{\beta_1\sqrt{t}}\int_{t^{-\frac{1}{2}}}^1\sin(\beta_1\sqrt{t}\sigma)\mathrm{e}^{-(\beta_0+\beta_2)\sigma^2}\left(-2(\beta_0+\beta_2)-\sigma^{-2}\right)\mathrm{d}\sigma\\
&=-\frac{2}{\beta_1\sqrt{t}}\sin(\beta_1\sqrt{t})\mathrm{e}^{-(\beta_0+\beta_2)}+\frac{2}{\beta_1}\sin(\beta_1)\mathrm{e}^{-(\beta_0+\beta_2)t^{-1}}\\
&\quad-\frac{2}{\beta_1\sqrt{t}}\int_{t^{-\frac{1}{2}}}^1\sin(\beta_1\sqrt{t}\sigma)\mathrm{e}^{-(\beta_0+\beta_2)\sigma^2}\left(2(\beta_0+\beta_2)+\sigma^{-2}\right)\mathrm{d}\sigma.
\end{align*}
It means
\begin{align*}
\left|-2\int_{t^{-\frac{1}{2}}}^1\frac{\cos(\beta_1\sqrt{t}\sigma)}{\sigma}\mathrm{e}^{-(\beta_0+\beta_2)\sigma^2}\mathrm{d}\sigma\right|\lesssim t^{-\frac{1}{2}}+1+t^{-\frac{1}{2}}\int_{t^{-\frac{1}{2}}}^1\mathrm{d}\sigma+t^{-\frac{1}{2}}\int_{t^{-\frac{1}{2}}}^1\sigma^{-2}\mathrm{d}\sigma\lesssim 1
\end{align*}
for $t\gg1$. Similarly, concerning $t\gg1$, we also arrive at
\begin{align*}
\left|\frac{1}{2}	\int_{t^{-\frac{1}{2}}}^1\frac{\cos(2\beta_1\sqrt{t}\sigma)}{\sigma}\mathrm{e}^{-2\beta_2\sigma^2}\mathrm{d}\sigma\right|\lesssim 1.
\end{align*}
Summarizing the last estimates, we say
\begin{align*}
	\|\ml{M}_2(t,|\xi|)\|_{L^2}^2\gtrsim \ln t-C\gtrsim \ln t
\end{align*}
for large-time $t\gg1$.\medskip

	\noindent\underline{Higher-dimensional case: $n=3$.} With analogous manner to \eqref{Take-1}, one can get
\begin{align*}
	\|\ml{M}_3(t,|\xi|)\|_{L^2}^2 &\gtrsim t^{-\frac{1}{2}}\int_{0}^{\infty}\left(\mathrm{e}^{-2\beta_0\sigma^2}+\frac{1}{2}\mathrm{e}^{-2\beta_2\sigma^2}\right)\mathrm{d}\sigma\notag\\
	&\quad+t^{-\frac{1}{2}}\int_{0}^{\infty}\left(-2\cos(\beta_1\sqrt{t}\sigma)\mathrm{e}^{-(\beta_0+\beta_2)\sigma^2}+\frac{1}{2}\cos(2\beta_1\sqrt{t}\sigma)\mathrm{e}^{-2\beta_2\sigma^2}\right)\mathrm{d}\sigma.
\end{align*}
On the other hand, the Riemann-Lebesgue theorem states that
\begin{align*}
	-2\int_{0}^{\infty}\cos(\beta_1\sqrt{t}\sigma)\mathrm{e}^{-(\beta_0+\beta_2)\sigma^2}\mathrm{d}\sigma+\frac{1}{2}\int_0^{\infty}\cos(2\beta_1\sqrt{t}\sigma)\mathrm{e}^{-2\beta_2\sigma^2}\mathrm{d}\sigma=o(1)
\end{align*}
as $t\gg1$, which yields
\begin{align*}
\|\ml{M}_3(t,|\xi|)\|_{L^2}^2\gtrsim t^{-\frac{1}{2}}+o(t^{-\frac{1}{2}})\gtrsim t^{-\frac{1}{2}}
\end{align*}
for $t\gg1$. The proof is totally complete.
\end{proof}
\begin{remark}
By the similar procedure as the one for $n=3$, the optimal estimates also hold for large-time such that $\|\ml{M}_n(t,|\xi|)\|_{L^2}\simeq t^{\frac{1}{2}-\frac{n}{4}}$ for any $n\geqslant 4$. But our goal is to study the thermoelasticity in physical dimensions $n=1,2,3$ only.
\end{remark}
With the aid of the recombination
\begin{align*}
	 u^{k,p_0}(t,x)-\ml{G}_0(t,x)P_{u_1^{k,p_0}}-\ml{G}_{1,k}(t,x)P_{\theta_0} &=\left( u^{k,p_0}(t,x)-\ml{G}_0(t,|D|)u_1^{k,p_0}-\ml{G}_{1,k}(t,D)\theta_0 \right) \\
	& \quad\  + \left( \ml{G}_0(t,|D|)u_1^{k,p_0}-\ml{G}_0(t,x)P_{u_1^{k,p_0}} \right)\\
	&\quad\  +\big( \ml{G}_{1,k}(t,D)\theta_0-\ml{G}_{1,k}(t,x)P_{\theta_0} \big),
\end{align*}
observing that 
\begin{align*}
	\left\|\big(1-\chi_{\intt}(D)\big)
	\left(u^{k,p_0}(t,\cdot)-\ml{G}_0(t,|D|)u_1^{k,p_0}(\cdot)-\ml{G}_{1,k}(t,D)\theta_0(\cdot)\right)\right\|_{L^2}\lesssim \mathrm{e}^{-ct}\left\|\left(u^{k,p_0}_0,u^{k,p_0}_1,\theta_0\right)\right\|_{(L^2)^3}
\end{align*}
and \eqref{eq:28}, we have 
\begin{align}
	&\left\| u^{k,p_0}(t,\cdot)-\ml{G}_0(t,\cdot)P_{u_1^{k,p_0}}-\ml{G}_{1,k}(t,\cdot)P_{\theta_0} \right\|_{L^2}\notag\\
	&\qquad\lesssim t^{-\frac{n}{4}}\left\|\left(u^{k,p_0}_0,u^{k,p_0}_1,\theta_0\right)\right\|_{(L^2 \cap L^{1})^3}+\left\|
	\ml{G}_0(t,|D|)u_1^{k,p_0}(\cdot)-\ml{G}_0(t,\cdot)P_{u_1^{k,p_0}}
	\right\|_{L^2} \notag \\
	&\qquad \quad +
	\left\|
	\ml{G}_{1,k}(t,D)\theta_0(\cdot)-\ml{G}_{1,k}(t,\cdot)P_{\theta_0} 
	\right\|_{L^2} \label{eq:36},
\end{align}
by some minor modifications of the derivation of \eqref{Sub-tract-01}.  On the other hand, just applying the same argument for $\ml{E}_2(t,x)$ below, particularly \eqref{Eq-01}, to the decomposition 
\begin{align*}
	 \ml{G}_0(t,|D|)u_1^{k,p_0}(x)-\ml{G}_0(t,x)P_{u_1^{k,p_0}} & =
	\int_{|y|\leqslant t^{\frac{1}{4}}}\big(\ml{G}_0(t,x-y)-\ml{G}_0(t,x)\big)u_1^{k,p_0}(y)\mathrm{d}y\\
	&\quad+\int_{|y|\geqslant t^{\frac{1}{4}}}\ml{G}_0(t,x-y)u_1^{k,p_0}(y)\mathrm{d}y-\int_{|y|\geqslant t^{\frac{1}{4}}}
	\ml{G}_0(t,x)u_1^{k,p_0}(y)\mathrm{d}y,
\end{align*}
we easily obtain 
\begin{align}
	\left\|
	 \ml{G}_0(t,|D|)u_1^{k,p_0}(\cdot)-\ml{G}_0(t,\cdot)P_{u_1^{k,p_0}}
	\right\|_{L^2}=o\big(\ml{A}_n(t)\big) \label{eq:37}
\end{align}
as $t \to \infty$.
By the same way we also have
\begin{align}
\left\|
\ml{G}_{1,k}(t,D)\theta_0(\cdot)-\ml{G}_{1,k}(t,\cdot)P_{\theta_0} 
\right\|_{L^2}=o\big(\ml{A}_n(t)\big) \label{eq:38}
\end{align}
as $t \to \infty$.
Thus, we arrive at the estimate  
\begin{align}
	\left\| u^{k,p_0}(t,\cdot)-\ml{G}_0(t,\cdot)P_{u_1^{k,p_0}}-\ml{G}_{1,k}(t,\cdot)P_{\theta_0} \right\|_{L^2}
	= o\big(\ml{A}_n(t)\big)  \label{eq:39}
\end{align}
for $t \gg 1$ and $u^{k,p_0}_0,u^{k,p_0}_1,\theta_0\in L^2\cap L^{1}$ by \eqref{eq:36}-\eqref{eq:38}.

 By directly applying Proposition \ref{Prop-Fourier-mul-Takeda} and \eqref{C-01} associated with corresponding constants
 \begin{align}\label{beta012}
 \beta_0=\frac{\kappa b^2}{b^2+\gamma_1\gamma_2},\ \ \beta_1=\sqrt{b^2+\gamma_1\gamma_2},\ \ \beta_2=\frac{\kappa\gamma_1\gamma_2}{2(b^2+\gamma_1\gamma_2)},
 \end{align}
  we claim for large-time that
\begin{align}\label{lower-bound-02}
\left\|\frac{1}{i}\widehat{\ml{G}}_{1,k}(t,\xi)\right\|_{L^2}\gtrsim \ml{A}_n(t).
\end{align}
Due to the real value of $\frac{1}{i}\widehat{\ml{G}}_{1,k}(t,\xi)$, by using \eqref{lower-bound-01}, \eqref{lower-bound-02}, one claims
\begin{align*}
\|\varphi^k(t,\cdot)\|_{L^2}^2&=\left\|\ml{G}_0(t,\cdot)P_{u_1^{k,p_0}}+i\frac{1}{i}\ml{G}_{1,k}(t,\cdot)P_{\theta_0}\right\|_{L^2}^2\\
&=\|\ml{G}_0(t,\cdot)\|_{L^2}^2|P_{u_1^{k,p_0}}|^2+\left\|\frac{1}{i}\ml{G}_{1,k}(t,\cdot)\right\|_{L^2}^2|P_{\theta_0}|^2\\
&\gtrsim\big(\ml{A}_n(t)\big)^2\left(|P_{u_1^{k,p_0}}|^2+|P_{\theta_0}|^2\right)
\end{align*}
for $t\gg1$. Recalling \eqref{eq:39} and Minkowski's inequality, for $t\gg1$ we can get
\begin{align*}
	\|u^{k,p_0}(t,\cdot)\|_{L^2}&\geqslant \|\varphi^k(t,\cdot)\|_{L^2} -\left\|u^{k,p_0}(t,\cdot)-\ml{G}_0(t,\cdot)P_{u_1^{k,p_0}}-\ml{G}_{1,k}(t,\cdot)P_{\theta_0}\right\|_{L^2}\\
	&\gtrsim\ml{A}_n(t)\sqrt{|P_{u_1^{k,p_0}}|^2+|P_{\theta_0}|^2}
	-t^{-\frac{n}{4}}\left\|\left(u_0^{k,p_0},u_1^{k,p_0},\theta_0\right)\right\|_{(L^{2} \cap L^{1})^3} -o\big(\ml{A}_n(t)\big)\\
	&\gtrsim\ml{A}_n(t)\sqrt{|P_{u_1^{k,p_0}}|^2+|P_{\theta_0}|^2}.
\end{align*}
It immediately completes our proof of \eqref{Estimates:Upper-Bound} from the below side.
 
\subsection{Second-order asymptotic profiles for the potential solution}\label{Sub-Second-Profile}
Firstly, let us explain shortly the idea of construction for the second-order profile $\psi^k=\psi^k(t,x)$ since the first-order profile (leading term) $\varphi^k=\varphi^k(t,x)$ has been built by extracting the dominant terms. To get the second-order profile, we not only derive the second-order expansions of solution, but also need to find the worst term with the help of Taylor's expansions.  Recalling the profiles $\varphi^k(t,x)$ and $\psi^k(t,x)$, we apply a suitable decomposition as follows:
\begin{align*}
u^{k,p_0}(t,x)-\varphi^k(t,x)-\psi^k(t,x)=\sum\limits_{j=1,\dots,6}\ml{E}_j(t,x),
\end{align*}
where these error functions are defined by
\begin{align*}
\ml{E}_1(t,x)&:=u^{k,p_0}(t,x)-\ml{G}_0(t,|D|)u_1^{k,p_0}-\ml{G}_{1,k}(t,D)\theta_0-\ml{G}_2(t,|D|)u_0^{k,p_0}\\
&\quad\ -\big(\ml{H}_0(t,|D|)+\ml{G}_3(t,|D|)\big)u_1^{k,p_0}-\big(\ml{H}_{1,k}(t,D)+\ml{G}_{4,k}(t,D)\big)\theta_0,
\end{align*}
and
\begin{align*}
\ml{E}_2(t,x)&:=\ml{G}_0(t,|D|)u_1^{k,p_0}-\ml{G}_0(t,x)P_{u_1^{k,p_0}}-\nabla\ml{G}_0(t,x)\circ M_{u_1^{k,p_0}},\\
\ml{E}_3(t,x)&:=\ml{G}_{1,k}(t,D)\theta_0-\ml{G}_{1,k}(t,x)P_{\theta_0}-\nabla\ml{G}_{1,k}(t,x)\circ M_{\theta_0},
\end{align*}
as well as
\begin{align*}
\ml{E}_4(t,x)&:=\ml{G}_2(t,|D|)u_0^{k,p_0}-\ml{G}_2(t,x)P_{u_0^{k,p_0}},\\
\ml{E}_5(t,x)&:=\big(\ml{H}_0(t,|D|)+\ml{G}_3(t,|D|)\big)u_1^{k,p_0}-\big(\ml{H}_0(t,x)+\ml{G}_3(t,x)\big)P_{u_1^{k,p_0}},\\
\ml{E}_6(t,x)&:=\big(\ml{H}_{1,k}(t,D)+\ml{G}_{4,k}(t,D)\big)\theta_0-\big(\ml{H}_{1,k}(t,x)+\ml{G}_{4,k}(t,x)\big)P_{\theta_0}.
\end{align*}

Let us review the final estimate in Proposition \ref{Prop-Pointwise}. The application of the triangle inequality immediately shows
\begin{align*}
&\chi_{\intt}(\xi)\left|\widehat{u}^{k,p_0}-\widehat{\ml{G}}_2\widehat{u}_0^{k,p_0}-\left(\widehat{\ml{G}}_0+\widehat{\ml{H}}_0+\widehat{\ml{G}}_3\right)\widehat{u}_1^{k,p_0}-\left(\widehat{\ml{G}}_{1,k}+\widehat{\ml{H}}_{1,k}+\widehat{\ml{G}}_{4,k}\right)\widehat{\theta}_0\right|\\
&\qquad\lesssim\chi_{\intt}(\xi)|\xi|\mathrm{e}^{-c|\xi|^2t}\left(|\widehat{u}_0^{k,p_0}|+|\widehat{u}_1^{k,p_0}|+|\widehat{\theta}_0|\right),
\end{align*}
where Propositions \ref{Prop-Improvement-J0} and \ref{Prop-Improvement-J1} were used. Thus, we split the goal into two frequencies parts to get
\begin{align*}
\|\ml{E}_1(t,\cdot)\|_{L^2}&\leqslant \|\chi_{\intt}(D)\ml{E}_1(t,\cdot)\|_{L^2}+\left\|\big(1-\chi_{\intt}(D)\big)\ml{E}_1(t,\cdot)\right\|_{L^2}\\
&\lesssim \left\|\chi_{\intt}(\xi)|\xi|\mathrm{e}^{-c|\xi|^2t}\left(|\widehat{u}_0^{k,p_0}|+|\widehat{u}_1^{k,p_0}|+|\widehat{\theta}_0|\right)\right\|_{L^2}+\mathrm{e}^{-ct}\left\|\left(u_0^{k,p_0},u_1^{k,p_0},\theta_0\right)\right\|_{(L^2)^3}\\
&\lesssim t^{-\frac{1}{2}-\frac{n}{4}}\left\|\left(u_0^{k,p_0},u_1^{k,p_0},\theta_0\right)\right\|_{(L^2\cap L^1)^3}
\end{align*}
for $t\gg1$. To deal with the second error term, we may decompose it into three functions as follows:
\begin{align*}
\ml{E}_2(t,x)&=\int_{|y|\leqslant t^{\alpha_1}}\big(\ml{G}_0(t,x-y)-\ml{G}_0(t,x)-y\circ\nabla\ml{G}_0(t,x)\big)u_1^{k,p_0}(y)\mathrm{d}y\\
&\quad+\int_{|y|\geqslant t^{\alpha_1}}\big(\ml{G}_0(t,x-y)-\ml{G}_0(t,x)\big)u_1^{k,p_0}(y)\mathrm{d}y+\int_{|y|\geqslant t^{\alpha_1}}(-y)\circ\nabla\ml{G}_0(t,x)u_1^{k,p_0}(y)\mathrm{d}y\\
&=:\ml{E}_{2,1}(t,x)+\ml{E}_{2,2}(t,x)+\ml{E}_{2,3}(t,x)
\end{align*}
carrying a small positive constant $\alpha_1\ll 1$. The application of Taylor's expansions implies
\begin{align}
	|\ml{G}_0(t,x-y)-\ml{G}_0(t,x)|&\lesssim|y|\,|\nabla\ml{G}_0(t,x-\sigma_0y)|,\label{Eq-01}\\
	|\ml{G}_0(t,x-y)-\ml{G}_0(t,x)-y\circ\nabla\ml{G}_0(t,x)|&\lesssim|y|^2|\nabla^2\ml{G}_0(t,x-\sigma_1y)|,\notag
\end{align}
with some constants $\sigma_0,\sigma_1\in(0,1)$. So, it implies
\begin{align*}
	\|\ml{E}_{2,1}(t,\cdot)\|_{L^2}&\lesssim t^{2\alpha_1}\|\,|\xi|^2\widehat{\ml{G}}_0(t,|\xi|)\|_{L^2}\|u_1^{k,p_0}\|_{L^1}\lesssim t^{2\alpha_1-\frac{1}{2}-\frac{n}{4}}\|u_1^{k,p_0}\|_{L^1},\\
	\|\ml{E}_{2,2}(t,\cdot)\|_{L^2}+\|\ml{E}_{2,3}(t,\cdot)\|_{L^2}&\lesssim\|\,|\xi|\widehat{\ml{G}}_0(t,|\xi|)\|_{L^2}\int_{|y|\geqslant t^{\alpha_1}}|y|\,|u_1^{k,p_0}(y)|\mathrm{d}y=o(t^{-\frac{n}{4}}),
\end{align*}
for $t\gg1$, where we considered $u_1^{k,p_0}\in L^{1,1}$ so that
\begin{align*}
\lim\limits_{t\to\infty}\int_{|y|\geqslant t^{\alpha_1}}|y|\,|u_1^{k,p_0}(y)|\mathrm{d}y=0.
\end{align*}
Thus, we claim
\begin{align*}
	\|\ml{E}_2(t,\cdot)\|_{L^2}\leqslant\sum\limits_{j=1,2,3}\|\ml{E}_{2,j}(t,\cdot)\|_{L^2}=o(t^{-\frac{n}{4}})
\end{align*}
for $t\gg1$ and $n=1,2,3$.

By the same way, with the aid of $\theta_0\in L^{1,1}$, we may obtain
\begin{align*}
\|\ml{E}_3(t,\cdot)\|_{L^2}=o(t^{-\frac{n}{4}})
\end{align*}
for $t\gg1$. Applying the same method as \eqref{Sub-tract-01}, we assert that
\begin{align*}
	\|\ml{E}_4(t,\cdot)\|_{L^2}+\|\ml{E}_5(t,\cdot)\|_{L^2}+\|\ml{E}_6(t,\cdot)\|_{L^2}\lesssim t^{-\frac{1}{2}-\frac{n}{4}}\left\|\left(u_0^{k,p_0},u_1^{k,p_0},\theta_0\right)\right\|_{(L^2\cap L^1)\times(L^2\cap L^{1,1})^2}
\end{align*}
for large-time. Summarizing all estimates in the above, that leads to \eqref{Second-order-expansion}  for $t\gg1$ to complete our proof of Corollary \ref{Coro-Second-order}.

\subsection{Optimal estimates for the leading term}\label{Sub-Leading}
Concerning upper bound estimates for large-time, we have proved them already in \eqref{Sub-tract-01} and \eqref{Sub-tract-large-01} totally. Thus, we are just required to concentrate on the lower one, whose key is lower bound estimates for $\psi^k(t,\cdot)$ in the $L^2$ norm. The Fourier imagine of $\psi^k(t,\cdot)$ is represented by
\begin{align*}
\widehat{\psi}^k(t,\xi)&=i\left(\big(\xi\circ M_{u_1^{k,p_0}}\big)\widehat{\ml{G}}_0(t,|\xi|)+(\xi\circ M_{\theta_0})\widehat{\ml{G}}_{1,k}(t,\xi)\right)+\widehat{\ml{G}}_2(t,|\xi|)P_{u_0^{k,p_0}}\\
&\quad+\big(\widehat{\ml{H}}_0(t,|\xi|)+\widehat{\ml{G}}_3(t,|\xi|)\big)P_{u_1^{k,p_0}}+\big(\widehat{\ml{H}}_{1,k}(t,\xi)+\widehat{\ml{G}}_{4,k}(t,\xi)\big)P_{\theta_0},
\end{align*}
therefore, we may re-organize it to separate the imaginary part and the real part
\begin{align*}
\widehat{\psi}^k(t,\xi)&=i\left(\big(\xi\circ M_{u_1^{k,p_0}}\big)\widehat{\ml{G}}_0(t,|\xi|)+\frac{1}{i}\big(\widehat{\ml{H}}_{1,k}(t,\xi)+\widehat{\ml{G}}_{4,k}(t,\xi)\big)P_{\theta_0}\right)\\
&\quad+\left(\widehat{\ml{G}}_2(t,|\xi|)P_{u_0^{k,p_0}}+\big(\widehat{\ml{H}}_0(t,|\xi|)+\widehat{\ml{G}}_3(t,|\xi|)\big)P_{u_1^{k,p_0}}+i(\xi\circ M_{\theta_0})\widehat{\ml{G}}_{1,k}(t,\xi)\right)\\
&=:i\widehat{\ml{E}}_7(t,\xi)+\widehat{\ml{E}}_8(t,\xi),
\end{align*}
where $\widehat{\ml{E}}_7(t,\xi)$ and $\widehat{\ml{E}}_8(t,\xi)$ are the real functions with parameters \eqref{beta012} such that
\begin{align*}
&\widehat{\ml{E}}_7(t,\xi)=\frac{\sin(\beta_1|\xi|t)}{\beta_1}\mathrm{e}^{-\beta_2|\xi|^2t}\left[ \left(\frac{\xi}{|\xi|}\circ M_{u_1^{k,p_0}}\right)+A_1\frac{\xi_k}{|\xi|}|\xi|^2tP_{\theta_0}+A_2\frac{\xi_k}{|\xi|}P_{\theta_0} \right],\\
&\widehat{\ml{E}}_8(t,\xi)=\mathrm{e}^{-\beta_0|\xi|^2t}\left[\frac{\gamma_1\gamma_2}{b^2+\gamma_1\gamma_2}P_{u_0^{k,p_0}}+\frac{\kappa\gamma_1\gamma_2}{(b^2+\gamma_1\gamma_2)^2}P_{u_1^{k,p_0}}+\frac{\gamma_1}{b^2+\gamma_1\gamma_2}\frac{\xi_k}{|\xi|}\left(\frac{\xi}{|\xi|}\circ M_{\theta_0}\right)\right]\\
&\quad\ \ +\cos(\beta_1|\xi|t)\mathrm{e}^{-\beta_2|\xi|^2t}\left[\frac{b^2P_{u_0^{k,p_0}}}{b^2+\gamma_1\gamma_2}-\frac{A_1}{\gamma_1}|\xi|^2tP_{u_1^{k,p_0}}-\frac{\kappa \gamma_1\gamma_2P_{u_1^{k,p_0}}}{(b^2+\gamma_1\gamma_2)^2}-\frac{\gamma_1}{b^2+\gamma_1\gamma_2}\frac{\xi_k}{|\xi|}\left(\frac{\xi}{|\xi|}\circ M_{\theta_0}\right)\right],
\end{align*}
equipping
	\begin{align*}
	A_1:=\frac{\kappa^2\gamma_1^2\gamma_2(\gamma_1\gamma_2+4b^2)}{8\beta_1^6}\ \  \mbox{and}\ \ A_2:=\frac{\gamma_1\kappa(\gamma_1\gamma_2-2b^2)}{\beta_1^4}.
	\end{align*}
Because $\widehat{\ml{E}}_7(t,\xi)$ and $\widehat{\ml{E}}_8(t,\xi)$ are the imaginary part and the real part, namely,
\begin{align*}
\|\widehat{\psi}^k(t,\xi)\|_{L^2}^2=\|\widehat{\ml{E}}_7(t,\xi)\|_{L^2}^2+\|\widehat{\ml{E}}_8(t,\xi)\|_{L^2}^2,
\end{align*}
 we will estimate them in the $L^2$ norm separately.

Let us apply polar coordinates again and separate them with respect to the power of $r$ in the integration to see
\begin{align*}
\|\widehat{\ml{E}}_7(t,\xi)\|_{L^2}^2&=\frac{1}{\beta_1^2}\int_{0}^{\infty}\int_{\mb{S}^{n-1}}|\sin(\beta_1rt)|^2\mathrm{e}^{-2\beta_2r^2t}r^{n-1}\left(\big(\omega\circ M_{u_1^{k,p_0}}\big)+A_1\omega_kr^2tP_{\theta_0}+A_2\omega_kP_{\theta_0}\right)^2\mathrm{d}\sigma_{\omega}\mathrm{d}r\\
&=\frac{\mb{B}_0}{\beta_1^2}\int_{0}^{\infty}|\sin(\beta_1rt)|^2\mathrm{e}^{-2\beta_2r^2t}r^{n-1}\mathrm{d}r+\frac{\mb{B}_1}{\beta_1^2}t\int_{0}^{\infty}|\sin(\beta_1rt)|^2\mathrm{e}^{-2\beta_2r^2t}r^{n+1}\mathrm{d}r\\
&\quad+\frac{\mb{B}_2}{\beta_1^2}t^2\int_{0}^{\infty}|\sin(\beta_1rt)|^2\mathrm{e}^{-2\beta_2r^2t}r^{n+3}\mathrm{d}r,
\end{align*}
where the integrals over the domains $\mb{S}^{n-1}$ are
\begin{align*}
\mb{B}_0&:=\int_{\mb{S}^{n-1}}\left(\big(\omega\circ M_{u_1^{k,p_0}}\big)^2+2A_2\omega_kP_{\theta_0}\big(\omega\circ M_{u_1^{k,p_0}}\big)+|A_2|^2\omega_k^2|P_{\theta_0}|^2\right)\mathrm{d}\sigma_{\omega},\\
\mb{B}_1&:=2A_1P_{\theta_0}\int_{\mb{S}^{n-1}}\left(\big(\omega\circ M_{u_1^{k,p_0}}\big)+A_2\omega_kP_{\theta_0}\right)\omega_k\mathrm{d}\sigma_{\omega},\\
\mb{B}_2&:=|A_1|^2|P_{\theta_0}|^2\int_{\mb{S}^{n-1}}\omega_k^2\mathrm{d}\sigma_{\omega}.
\end{align*}
By using the property of the basic as follows:
\begin{align*}
\int_{\mb{S}^{n-1}}\omega_j\omega_k\mathrm{d}\sigma_{\omega}=\begin{cases}
	\displaystyle{\frac{1}{n}|\mb{S}^{n-1}|}&\mbox{if}\ \ j=k,\\
	0&\mbox{if}\ \ j\neq k,
\end{cases}
\end{align*}
we are able to explicitly calculate the constants
\begin{align*}
\mb{B}_0&=\frac{|\mb{S}^{n-1}|}{n}\left(|M_{u_1^{k,p_0}}|^2+2A_2P_{\theta_0}M_{u_1^{k,p_0}}^k+|A_2|^2|P_{\theta_0}|^2\right),\\
\mb{B}_1&=2A_1P_{\theta_0}\frac{|\mb{S}^{n-1}|}{n}\left(M_{u_1^{k,p_0}}^k+A_2P_{\theta_0}\right),\\
\mb{B}_2&=|A_1|^2|P_{\theta_0}|^2\frac{|\mb{S}^{n-1}|}{n},
\end{align*}
with $M_{u_1^{k,p_0}}^k$ denoting the $k$-th element of the vector $M_{u_1^{k,p_0}}=(M_{u_1^{k,p_0}}^1,\cdots,M_{u_1^{k,p_0}}^n)$. Indeed, for any $\alpha_2\geqslant0$, taking $s=r\sqrt{t}$ and $\eta=s\sqrt{2\beta_2}$, we know from the Riemann-Lebesgue theorem that
\begin{align*}
&\int_0^{\infty}|\sin(\beta_1rt)|^2\mathrm{e}^{-2\beta_2r^2t}r^{n-1+\alpha_2}\mathrm{d}r\\
&\qquad=t^{-\frac{n+\alpha_2}{2}}\int_0^{\infty}|\sin(\beta_1s\sqrt{t})|^2\mathrm{e}^{-2\beta_2s^2}s^{n-1+\alpha_2}\mathrm{d}s\\
&\qquad=\frac{1}{2}t^{-\frac{n+\alpha_2}{2}}\left(\,\int_0^{\infty}\mathrm{e}^{-2\beta_2s^2}s^{n-1+\alpha_2}\mathrm{d}s-\int_0^{\infty}\cos(2\beta_1s\sqrt{t})\mathrm{e}^{-2\beta_2s^2}s^{n-1+\alpha_2}\mathrm{d}s\right)\\
&\qquad=\frac{1}{2} t^{-\frac{n+\alpha_2}{2}}\left((2\beta_2)^{-\frac{n+\alpha_2}{2}}\int_0^{\infty}\mathrm{e}^{-\eta^2}\eta^{n-1+\alpha_2}\mathrm{d}\eta+o(1)\right)\\
&\qquad=\frac{1}{4} t^{-\frac{n+\alpha_2}{2}}\left((2\beta_2)^{-\frac{n+\alpha_2}{2}}\Gamma\left(\frac{n+\alpha_2}{2}\right)+o(1)\right)
\end{align*}
for $t\gg1$, where we employed $2\sin^2z=1-\cos(2z)$ and
\begin{align*}
	\Gamma(z):=\int_0^{\infty}\mathrm{e}^{-\eta}\eta^{z-1}\mathrm{d}\eta=2\int_0^{\infty}\mathrm{e}^{-\eta^2}\eta^{2z-1}\mathrm{d}\eta.
\end{align*}
Due to the relation $\Gamma(z+1)=z \Gamma(z)$ so that
\begin{align*}
\Gamma\left(\frac{n+4}{2}\right)=\frac{n+2}{2}\Gamma\left(\frac{n+2}{2}\right)=\frac{(n+2)n}{4}\Gamma\left(\frac{n}{2}\right),
\end{align*}
it shows
\begin{align*}
\|\widehat{\ml{E}}_7(t,\xi)\|_{L^2}^2&=\frac{(2\beta_2)^{-\frac{n}{2}}}{4\beta_1^2}t^{-\frac{n}{2}}\left[\mb{B}_0\Gamma\left(\frac{n}{2}\right)+(2\beta_2)^{-1}\mb{B}_1\Gamma\left(\frac{n+2}{2}\right)+(2\beta_2)^{-2}\mb{B}_2\Gamma\left(\frac{n+4}{2}\right)+o(1)\right]\\
&\gtrsim t^{-\frac{n}{2}}\left(\mb{B}_0+\frac{n}{4\beta_2}\mb{B}_1+\frac{(n+2)n}{16 \beta_2^2}\mb{B}_2 \right)=:
t^{-\frac{n}{2}}\frac{|\mb{S}^{n-1}|}{n}\widetilde{\mb{B}}^2.
\end{align*}
It is worth noting that the non-negative quantity $\widetilde{\mb{B}}^2$ can be rewritten by
\begin{align*}
\widetilde{\mb{B}}^2&=|M_{u_1^{k,p_0}}|^2+2A_2P_{\theta_0}M_{u_1^{k,p_0}}^k+|A_2|^2|P_{\theta_0}|^2+\frac{n}{2 \beta_2}A_1P_{\theta_0}\big(M_{u_1^{k,p_0}}^k+A_2P_{\theta_0}\big)\\
&\quad+\frac{(n+2)n}{16 \beta_2^2}|A_1|^2|P_{\theta_0}|^2\\
&=|M_{u_1^{k,p_0}}|^2+2 \left( A_2+\frac{n}{4 \beta_2}A_1 \right) P_{\theta_0} M_{u_1^{k,p_0}}^k 
+ \left[\left( A_2+\frac{n}{4 \beta_2}A_1 \right)^{2}+\frac{n}{8 \beta_{2}^{2}}A_{1}^{2}\right] | P_{\theta_0}|^{2},
\end{align*}
and then
\begin{align*}
\widetilde{\mb{B}}^2&=\left(
1-\frac{\left( A_2+\frac{n}{4 \beta_2}A_1 \right)^{2} }{
\left( A_2+\frac{n}{4 \beta_2}A_1 \right)^{2}  +\frac{n}{16 \beta_{2}^{2}} A_{1}^{2} 
} 
\right) |M_{u_1^{k,p_0}}^k|^{2} + \sum\limits_{j\neq k}|M^j_{u_1^{k,p_0}}|^2
+\frac{n}{16 \beta_{2}^{2}} A_{1}^{2} |P_{\theta_0}|^2 \\
& \quad \ + 
\left[\sqrt{
\left( A_2+\frac{n}{4 \beta_2}A_1 \right)^{2}  +\frac{n}{16 \beta_{2}^{2}} A_{1}^{2} 
}\, P_{\theta_{0}} +\frac{A_2+\frac{n}{4 \beta_2}A_1}{\sqrt{
\left( A_2+\frac{n}{4 \beta_2}A_1 \right)^{2}  +\frac{n}{16 \beta_{2}^{2}} A_{1}^{2} 
}}M_{u_1^{k,p_0}}^k\right]^{2}
\\
& \geqslant \frac{n}{16 \beta_{2}^{2}} A_{1}^{2} 
\min \left( 
1, 
\frac{1}{
\left( A_2+\frac{n}{4 \beta_2}A_1 \right)^{2}  +\frac{n}{16 \beta_{2}^{2}} A_{1}^{2} 
} 
\right)  \left( |M_{u_1^{k,p_0}}|^{2} +|P_{\theta_0}|^2 \right).
\end{align*}

Next, we turn to some estimates for $\|\widehat{\ml{E}}_8(t,\xi)\|_{L^2}^2$. For the simplicity, we may denote 
\begin{align*}
A_{3} & := \frac{\gamma_1\gamma_2P_{u_0^{k,p_0}}}{b^2+\gamma_1\gamma_2}, \quad 
A_{4}:= \frac{\kappa\gamma_1\gamma_2P_{u_1^{k,p_0}}}{(b^2+\gamma_1\gamma_2)^2}, \quad
A_{5}:= \frac{\gamma_1}{b^2+\gamma_1\gamma_2}, \\
A_{6} & :=  \frac{b^2P_{u_0^{k,p_0}}}{b^2+\gamma_1\gamma_2}, \quad \,\,
A_{7}:= -\frac{A_1}{\gamma_1}P_{u_1^{k,p_0}},
\end{align*}
in other words, the goal can be reset by
\begin{align*}
\widehat{\ml{E}}_8(t,\xi)&=\mathrm{e}^{-\beta_0|\xi|^2t}\left[A_3+A_4+A_5\frac{\xi_k}{|\xi|}\left(\frac{\xi}{|\xi|}\circ M_{\theta_0}\right)\right]\\
&\quad+\cos(\beta_1|\xi|t)\mathrm{e}^{-\beta_2|\xi|^2t}\left[A_6+A_7|\xi|^2t-A_4-A_5\frac{\xi_k}{|\xi|}\left(\frac{\xi}{|\xi|}\circ M_{\theta_0}\right)\right].
\end{align*}
Then, we divide it into three portions
\begin{align*}
\|\widehat{\ml{E}}_8(t,\xi)\|_{L^2}^2&=\mb{B}_3+\mb{B}_4+2\mb{B}_5,
\end{align*}
where 
\begin{align*}
\mb{B}_3&:=
\int_{0}^{\infty} \mathrm{e}^{-2\beta_0 r^2t}r^{n-1} 
\int_{\mb{S}^{n-1}}\left[A_{3}+A_{4}+A_{5} \omega_{k} \big(\omega\circ M_{\theta_{0}}\big) \right]^{2} \mathrm{d}\sigma_{\omega} \mathrm{d}r,\\
\mb{B}_4&:=\int_{0}^{\infty} |\cos(\beta_1rt)|^2 \mathrm{e}^{-2\beta_2r^2t} r^{n-1} 
\int_{\mb{S}^{n-1}}\left[A_{6}+A_{7}tr^{2}-A_{4}-A_{5} \omega_{k} \big(\omega\circ M_{\theta_{0}}\big) \right]^{2} \mathrm{d}\sigma_{\omega} \mathrm{d}r,
\end{align*}
and
\begin{align*}
\mb{B}_5 &:= \int_{0}^{\infty} \cos(\beta_1rt) \mathrm{e}^{-(\beta_0 +\beta_2)r^2t} r^{n-1} \\
& \quad\ \qquad \times 
\int_{\mb{S}^{n-1}}
\left[A_{3}+A_{4}+A_{5} \omega_{k} \big(\omega\circ M_{\theta_{0}}\big) \right]
\left[A_{6}+A_{7}tr^{2}-A_{4}-A_{5} \omega_{k} \big(\omega\circ M_{\theta_{0}}\big) \right] \mathrm{d}\sigma_{\omega} \mathrm{d}r,
\end{align*}
by the same argument for $\mb{B}_j$ when $j=0,1,2$. 
Now, let us recall the equality that 
\begin{align} \label{eq:35}
\int_{\mb{S}^{n-1}}\left| \omega_{k} \big(\omega\circ M_{\theta_{0}}\big) \right|^{2} \mathrm{d}\sigma_{\omega}
&=\frac{\pi^{\frac{n}{2}}}{2\Gamma(\frac{n+4}{2})}\sum\limits_{j\neq k}|M_{\theta_0}^j|^2+\frac{3\pi^{\frac{n}{2}}}{2\Gamma(\frac{n+4}{2})}|M_{\theta_0}^k|^2\notag\\
&=
\frac{|\mb{S}^{n-1}|}{n(n+2)}  \left( |M_{\theta_{0}}|^{2} + 2|M_{\theta_{0}}^{k}|^{2} \right).
\end{align}
Concerning the proof of \eqref{eq:35}, one may see \cite[Lemma 2.6]{Takeda=2022}. 
Combining the method of the derivations of $\mb{B}_0,\mb{B}_1$, as well as \eqref{eq:35}, we see that
\begin{align*}
\mb{B}_3&=|\mb{S}^{n-1}|
\int_{0}^{\infty} \mathrm{e}^{-2\beta_0 r^2t}r^{n-1} \left[
(A_{3} +A_{4} )^{2} +\frac{2}{n}(A_{3}+A_{4})A_{5} M_{\theta_{0}}^{k} + \frac{1}{n(n+2)} A_{5}^{2}\left(|M_{\theta_{0}}|^{2} + 2|M_{\theta_{0}}^{k}|^{2} \right)\right] \mathrm{d}r.
\end{align*}
Then, noting that $\frac{1}{n}\leqslant\frac{2}{n+2}$ for $n \geqslant2$ and the trivial case $M_{\theta_{0}}=M_{\theta_{0}}^k$ when $n=1$,
we conclude that
\begin{align*}
\mb{B}_3&\gtrsim t^{-\frac{n}{2}}\left( 
(A_{3}+A_{4})^{2} +|M_{\theta_{0}}|^{2}
\right)
\\
&\gtrsim t^{-\frac{n}{2}}\left( 
\left((b^{2}+ \gamma_{1} \gamma_{2}) P_{u_0^{k,p_0}} +\kappa  P_{u_1^{k,p_0}} \right)^{2} +|M_{\theta_{0}}|^{2}
\right)
\end{align*}
by applying the same argument for $\widetilde{\mb{B}}^2$. 

For the term $\mb{B}_4$, we can use $2\cos^2z=1+\cos(2z)$ to arrive at 
\begin{align*}
\mb{B}_4&=\mb{B}_{4,1}+\mb{B}_{4,2},
\end{align*}
where we wrote
\begin{align*}
\mb{B}_{4,1}&:=\frac{1}{2} \int_{0}^{\infty} \mathrm{e}^{-2\beta_2r^2t} r^{n-1} 
\int_{\mb{S}^{n-1}}\left[A_{6}+A_{7}tr^{2}-A_{4}-A_{5} \omega_{k} \big(\omega\circ M_{\theta_{0}}\big) \right]^{2} \mathrm{d}\sigma_{\omega} \mathrm{d}r, \\
\mb{B}_{4,2}&:=\frac{1}{2} \int_{0}^{\infty} \cos(2 \beta_1rt) \mathrm{e}^{-2\beta_2r^2t} r^{n-1} 
\int_{\mb{S}^{n-1}}\left[A_{6}+A_{7}tr^{2}-A_{4}-A_{5} \omega_{k} \big(\omega\circ M_{\theta_{0}}\big) \right]^{2} \mathrm{d}\sigma_{\omega} \mathrm{d}r.
\end{align*}
As in the estimate of $\mb{B}_{3}$, we re-formula 
\begin{align*}
\mb{B}_{4,1}&=t^{-\frac{n}{2}}\frac{|\mb{S}^{n-1}|}{2}
\int_{0}^{\infty} \mathrm{e}^{-2\beta_0 r^2}r^{n-1} \bigg[
(A_{6} -A_{4}+A_{7} r^{2} )^{2}
-\frac{2}{n}(A_{6} -A_{4}+A_{7} r^{2} ) A_{5} M_{\theta_{0}}^{k} \\
&\qquad\qquad\qquad\qquad\qquad\qquad \ \  \ + \frac{A_{5}^{2}}{n(n+2)} \left( |M_{\theta_{0}}|^{2} + 2|M_{\theta_{0}}^{k}|^{2} \right) \bigg] \mathrm{d}r
\end{align*}
and estimate
\begin{align*}
\mb{B}_{4,1} & \gtrsim t^{-\frac{n}{2}}
\int_{0}^{\infty} \mathrm{e}^{-2\beta_0 r^2}r^{n-1} \left(
(A_{6} -A_{4}+A_{7} r^{2} )^{2}+ |M_{\theta_{0}}|^{2} \right) \mathrm{d}r \\
& = t^{-\frac{n}{2}} \Gamma\left( \frac{n}{2} \right)(2 \beta_{2})^{-\frac{n}{2}} 
\left(
(A_{6}-A_{4})^{2} +\frac{n}{2 \beta_{2}} A_{7}(A_{6}-A_{4}) +\frac{n}{2} \left( \frac{n}{2}+1 \right) \frac{A_{7}^{2}}{(2 \beta_{2})^{2}} 
\right) \\
& \quad +t^{-\frac{n}{2}} |M_{\theta_{0}}|^{2}
\int_{0}^{\infty} \mathrm{e}^{-2\beta_0 r^2}r^{n-1} \mathrm{d}r \\
& \gtrsim t^{-\frac{n}{2}} \left( 
\left(A_{6} -A_{4} \right)^{2}+A_{7}^{2}+ |M_{\theta_{0}}|^{2}\right) \\
& \gtrsim t^{-\frac{n}{2}} \big( |P_{u_0^{k,p_0}}|^{2} +  |P_{u_1^{k,p_0}}|^{2} + |M_{\theta_{0}}|^{2} \big), 
\end{align*} 
where we used the fact that 
$(x-y)^{2} +y^{2} =0$ for $x,y \in \mathbb{R}$ implies $x=y=0$.
On the other hand, the remainder term $\mb{B}_{4,2}+2\mb{B}_{5}$ is easily estimated by the Riemann-Lebesgue formula as follows:
\begin{align*}
\mb{B}_{4,2}+2\mb{B}_{5} &=
\frac{t^{-\frac{n}{2}}}{2} \int_{0}^{\infty} \cos(2 \beta_1r \sqrt{t}) \mathrm{e}^{-2\beta_2r^2} r^{n-1} 
\int_{\mb{S}^{n-1}}\left[A_{6}+A_{7}r^{2}-A_{4}-A_{5} \omega_{k} \big(\omega\circ M_{\theta_{0}}\big) \right]^{2} \mathrm{d}\sigma_{\omega} \mathrm{d}r \\
& \quad+2t^{-\frac{n}{2}}
\int_{0}^{\infty} \cos(\beta_1r \sqrt{t}) \mathrm{e}^{-(\beta_0 +\beta_2)r^2} r^{n-1}\\
&\qquad\qquad\qquad\times\int_{\mb{S}^{n-1}}
\left[A_{3}+A_{4}+A_{5} \omega_{k} \big(\omega\circ M_{\theta_{0}}\big) \right]\\
&\qquad\qquad\qquad\qquad\qquad\times\left[A_{6}+A_{7}r^{2}-A_{4}-A_{5} \omega_{k} \big(\omega\circ M_{\theta_{0}}\big) \right] \mathrm{d}\sigma_{\omega} \mathrm{d}r\\
& =  o(t^{-\frac{n}{2}})
\end{align*}
as $t\gg1$ for $n=1,2,3$.

Thus, the summary of last estimates concludes
\begin{align*}
	\|\widehat{\psi}^k(t,\xi)\|_{L^2}^2
	&\gtrsim t^{-\frac{n}{2}}\big(|M_{u_1^{k,p_0}}|^{2} +|P_{\theta_0}|^2\big)+t^{-\frac{n}{2}} \big( |P_{u_0^{k,p_0}}|^{2} +  |P_{u_1^{k,p_0}}|^{2} + |M_{\theta_{0}}|^{2} \big)\\
	&\gtrsim t^{-\frac{n}{2}}\mb{B}^2.
\end{align*}
We combine with Minkowski's inequality and the derived error estimates \eqref{Second-order-expansion} so that
\begin{align*}
\|u^{k,p_0}(t,\cdot)-\varphi^k(t,\cdot)\|_{L^2}&\gtrsim t^{-\frac{n}{4}}|\mb{B}|-o(t^{-\frac{n}{4}})\gtrsim t^{-\frac{n}{4}}|\mb{B}|
\end{align*}
for $t\gg1$. Then, we complete the proof of Theorem \ref{Thm-Optimal-Lead}.

\section{Final remarks}\label{Sect-Final}
Throughout this work, we have investigated large-time asymptotic profiles for a typical hyperbolic-parabolic coupled system, i.e. the classical thermoelastic system for $n=1,2,3$. Our methods are based on the treatments for third-order (in time) evolution equations and fine analyses for the Fourier multipliers carrying singularities, dissipations and oscillations. We conjecture that our methodology can be generalized to other thermoelastic models even hyperbolic-hyperbolic coupled systems, including thermoelasticity system with second sound \cite{Racke=2002,Racke=2003} (i.e. elastic body with Cattaneo's law of heat conduction), thermoelasticity of type II or type III \cite{Reissig-Wang=2005,Yang-Wang=2006} (i.e. elastic body with Jeffreys type of heat conduction). Nevertheless, one needs a suitable way to deal with fourth-order evolution equations.

%
\appendix
\section{Lower bound estimates for the free wave equation}\label{Appendix-A}
This appendix contributes to lower bound estimates of the solution to the linear wave model \eqref{Wave-Eq} under $w_1\in L^1$ regularity, which improves the estimates in \cite[Theorems 1.1 and 1.2]{Ikehata=2022-wave}. To be specific, we only require $L^1$ regularity rather than $L^{1,1}$ regularity in \cite{Ikehata=2022-wave} for the second data. Moreover, we do not rely on the trick of shrinking domain in \cite[Equation (2.17)]{Ikehata=2022-wave}. 

The crucial tool is estimating the oscillating Fourier multiplier in the next proposition. 
\begin{prop}\label{Prop-Improved-Ikehata}
	Let $g\in L^1$ for $n=1,2$. The following lower bound estimates hold:
	\begin{align*}
	I_g(t;n):=\left\|\ml{F}^{-1}_{\xi\to x}\left(\frac{\sin(|\xi|t)}{|\xi|}\widehat{g}(\xi)\right)\right\|_{L^2}\gtrsim\begin{cases}
	\sqrt{t}\,|P_g|&\mbox{if}\ \ n=1,\\
	\sqrt{\ln t}\,|P_g|&\mbox{if}\ \ n=2,
	\end{cases}
	\end{align*}
for any $t\gg1$.
\end{prop}
\begin{proof}
When $n=1$, the application of change of variable $\eta=\xi t$ shows
\begin{align}\label{Eq-*1}
\big(I_g(t;1)\big)^2=t\int_{\mb{R}}\frac{|\sin\eta|^2}{|\eta|^2}\left|\widehat{g}\left(t^{-1}\eta \right)\right|^2\mathrm{d}\eta.
\end{align}
By the Lebesgue dominated convergence theorem and $\widehat{g}(0)=P_g$, we know
\begin{align*}
\lim\limits_{t\to\infty}\int_{\mb{R}}\frac{|\sin\eta|^2}{|\eta|^2}\left|\widehat{g}\left(t^{-1}\eta\right)\right|^2\mathrm{d}\eta&=\int_{\mb{R}}\frac{|\sin\eta|^2}{|\eta|^2}|\widehat{g}(0)|^2\mathrm{d}\eta=|P_g|^2\int_{\mb{R}}\frac{|\sin\eta|^2}{|\eta|^2}\mathrm{d}\eta,
\end{align*}
namely,
\begin{align*}
\int_{\mb{R}}\frac{|\sin\eta|^2}{|\eta|^2}\left|\widehat{g}\left(t^{-1}\eta\right)\right|^2\mathrm{d}\eta-|P_g|^2\int_{\mb{R}}\frac{|\sin\eta|^2}{|\eta|^2}\mathrm{d}\eta=o(1)
\end{align*}
for $t\gg1$. So, one can get
\begin{align*}
	\mbox{RHS of \eqref{Eq-*1}}=t|P_g|^2\int_{\mb{R}}\frac{|\sin\eta|^2}{|\eta|^2}\mathrm{d}\eta+o(t)\gtrsim t|P_g|^2
\end{align*}
for $t\gg1$, due to the fact that $\int_{\mb{R}}|\eta|^{-2}|\sin\eta|^2\mathrm{d}\eta$ is a positive constant.

For the case in the two-dimension $n=2$, we apply polar coordinates and $\eta=\sqrt{t}r$ to find
\begin{align}\label{Eq-*2}
\big(I_g(t;2)\big)^2&=\int_0^{\infty}\frac{|\sin(rt)|^2}{r}\int_{\mb{S}}|\widehat{g}(\omega r)|^2\mathrm{d}\sigma_{\omega}\mathrm{d}r\notag\\
	&\geqslant\int_0^{\infty}\mathrm{e}^{-r^2t}\frac{|\sin(rt)|^2}{r}\int_{\mb{S}}|\widehat{g}(\omega r)|^2\mathrm{d}\sigma_{\omega}\mathrm{d}r\notag\\
	&\geqslant\int_0^{\infty}\mathrm{e}^{-\eta^2}\frac{|\sin(\sqrt{t}\eta)|^2}{\eta}\int_{\mb{S}}\left|\widehat{g}\left(t^{-\frac{1}{2}}\omega\eta\right)\right|^2\mathrm{d}\sigma_{\omega}\mathrm{d}\eta,
\end{align}
where we considered $\mathrm{e}^{-r^2t}\leqslant 1$ for any $t\geqslant 0$. Again, employing the well-known Lebesgue dominated convergence theorem and $\widehat{g}(0)=P_g$, it holds
\begin{align*}
\lim\limits_{t\to\infty}\int_{\mb{S}}\left|\widehat{g}\left(t^{-\frac{1}{2}}\omega\eta\right)\right|^2\mathrm{d}\sigma_{\omega}=\int_{\mb{S}}|\widehat{g}(0)|^2\mathrm{d}\sigma_{\omega}=|P_g|^2|\mb{S}|,
\end{align*}
in other words, concerning $t\gg1$,
\begin{align*}
	\int_{\mb{S}}\left|\widehat{g}\left(t^{-\frac{1}{2}}\omega\eta\right)\right|^2\mathrm{d}\sigma_{\omega}\geqslant\frac{1}{2}|P_g|^2|\mb{S}|.
\end{align*}
Consequently, the shrinking of domain into $[t^{-\frac{1}{2}},1]$ implies
\begin{align*}
	\mbox{RHS of \eqref{Eq-*2}}
	&\geqslant\frac{1}{2}|P_g|^2|\mb{S}|\int_{t^{-\frac{1}{2}}}^1\mathrm{e}^{-\eta^2}\frac{|\sin(\sqrt{t}\eta)|^2}{\eta}\mathrm{d}\eta\\
	&=\frac{1}{4}|P_g|^2|\mb{S}|\left(\int_{t^{-\frac{1}{2}}}^1\mathrm{e}^{-\eta^2}\frac{1}{\eta}\mathrm{d}\eta-\int_{t^{-\frac{1}{2}}}^1\mathrm{e}^{-\eta^2}\frac{\cos(2\sqrt{t}\eta)}{\eta}\mathrm{d}\eta\right).
\end{align*}
For one thing, a direct consequence indicates
\begin{align*}
	\int_{t^{-\frac{1}{2}}}^1\mathrm{e}^{-\eta^2}\frac{1}{\eta}\mathrm{d}\eta\geqslant\mathrm{e}^{-1}\int_{t^{-\frac{1}{2}}}^1\frac{1}{\eta}\mathrm{d}\eta=\frac{1}{2\mathrm{e}}\ln t
\end{align*}
for $t\gg1$, and an integration by parts leads to
\begin{align*}
\int_{t^{-\frac{1}{2}}}^1\mathrm{e}^{-\eta^2}\frac{\cos(2\sqrt{t}\eta)}{\eta}\mathrm{d}\eta=\left(\mathrm{e}^{-\eta^2}\frac{\sin(2\sqrt{t}\eta)}{2\sqrt{t}\eta}\right)\Big|_{\eta=t^{-\frac{1}{2}}}^{\eta=1}+\int_{t^{-\frac{1}{2}}}^1\frac{\sin(2\sqrt{t}\eta)}{2\sqrt{t}}\mathrm{e}^{-\eta^2}\left(2+\frac{1}{\eta^2}\right)\mathrm{d}\eta.
\end{align*}
It immediately gives
\begin{align*}
\left|\int_{t^{-\frac{1}{2}}}^1\mathrm{e}^{-\eta^2}\frac{\cos(2\sqrt{t}\eta)}{\eta}\mathrm{d}\eta\right|\leqslant \frac{1}{\sqrt{t}}+1+\frac{1}{\sqrt{t}}\int_0^1\mathrm{d}\eta+\frac{1}{2\sqrt{t}}\int_{t^{-\frac{1}{2}}}^1\eta^{-2}\mathrm{d}\eta\lesssim 1
\end{align*}
as large-time $t\gg1$. Finally, we conclude that
\begin{align*}
\mbox{RHS of \eqref{Eq-*2}}\gtrsim \ln t\,|P_g|^2-c\gtrsim \ln t\,|P_g|^2
\end{align*}
for $t\gg1$, which demonstrates the desired estimate for $n=2$.
\end{proof}
\begin{coro}\label{Coro-Appendix}
Let us consider the free wave equation \eqref{Wave-Eq} for $n=1,2$ carrying initial datum $w_0\in L^2$ and $w_1\in L^1$. Then, the solution $w$ satisfies the following lower bound estimates:
\begin{align*}
\|w(t,\cdot)\|_{L^2}\gtrsim\begin{cases}
	\sqrt{t}\,|P_{w_1}|&\mbox{if}\ \ n=1,\\
	\sqrt{\ln t}\,|P_{w_1}|&\mbox{if}\ \ n=2,
\end{cases}
\end{align*}
for any $t\gg1$.
\end{coro}
\begin{remark}
This result implies that the solution of free wave equation grows to infinite as $t\to\infty$ providing that $|P_{w_1}|\neq0$. Due to the upper bound estimates in \cite{Ikehata=2022-wave}, we may state the optimal growth estimates $\|w(t,\cdot)\|_{L^2}\simeq\ml{A}_n(t)$ for $t\gg1$ and $n=1,2$ by taking initial datum in $ L^2\cap L^1$ and $|P_{w_1}|\neq0$.
\end{remark}
\begin{proof}
It is well-known that the solution of \eqref{Wave-Eq} in the Fourier space is expressed by
\begin{align*}
\widehat{w}=\cos(|\xi|t)\widehat{w}_0+\frac{\sin(|\xi|t)}{|\xi|}\widehat{w}_1.
\end{align*}
Hence, according to $2|f+g|^2\geqslant|f|^2-2|g|^2$ and Proposition \ref{Prop-Improved-Ikehata}, we get
\begin{align*}
	\|\widehat{w}(t,\xi)\|_{L^2}^2&\geqslant\frac{1}{2}\left\|\ml{F}^{-1}_{\xi\to x}\left(\frac{\sin(|\xi|t)}{|\xi|}\widehat{w}_1\right)\right\|_{L^2}^2-\left\|\ml{F}^{-1}_{\xi\to x}\big(\cos(|\xi|t)\widehat{w}_0\big)\right\|_{L^2}^2\\
	&\gtrsim\begin{cases}
		t|P_{w_1}|^2-\|w_0\|_{L^2}^2&\mbox{if}\ \ n=1,\\
		\ln t\,|P_{w_1}|^2-\|w_0\|_{L^2}^2&\mbox{if}\ \ n=2,
	\end{cases}
\end{align*}
for $t\gg1$. Because of $w_0\in L^2$ as well as $w_1\in L^1$, our proof is completed.
\end{proof}
\section*{Acknowledgments}
The second author (Hiroshi Takeda) is supported in part by the Grant-in-Aid for Scientific Research (C)
(No. 19K03596) from Japan Society for the Promotion of Science.  The authors thank Xiufang Cui (Shanghai Jiao Tong University) for some discussions on elasticity in 2D and 3D.

\end{document}